\documentclass{article}

\usepackage{amsmath,amsthm,titling,enumitem,amssymb,etoolbox,mathrsfs,graphicx,xcolor,cancel,tikz-cd,hyperref,mathtools}

\hypersetup{
    colorlinks=true,
    citecolor=black,
    linkcolor=blue}

\usepackage[margin=2.5cm]{geometry}

\def\N{\mathbb{N}}
\def\Z{\mathbb{Z}}

\def\R{\mathbb{R}}

\def\tr{\operatorname{Tr}}

\def\Div{\operatorname{div}}

\def\spt{\operatorname{spt}}

\def\perm{\operatorname{Perm}}

\def\T{\mathbb{T}}

\def\@dashint#1#2{%
    {\setbox0=\hbox{$#1{#2}{\int}$ }
    \vcenter{\hbox{$#2-$ }}\kern-.9\wd0}%
}
\def\dashint{\mathchoice
    {\@dashint\displaystyle\textstyle}%
    {\@dashint\textstyle\scriptstyle}%
    {\@dashint\scriptstyle\scriptscriptstyle}%
    {\@dashint\scriptscriptstyle\scriptscriptstyle}%
    \!\int}

\newtheoremstyle{theoremstyle}
  {6pt}{15pt}
  {\itshape}
  {}
  {\bf}
  {.}
  {1em}
  {}

\newtheoremstyle{examplestyle}
  {6pt}{15pt}
  {}
  {}
  {\bf}
  {.}
  {1em}
  {}
  
\theoremstyle{theoremstyle}
\newtheorem{thm}{Theorem}[section]
\newtheorem{lemma}[thm]{Lemma}

\theoremstyle{examplestyle}

\newtheorem{remark}[thm]{Remark}

\numberwithin{equation}{section}

\date{}
\title{A Hölder estimate for the trajectories of the Navier-Stokes equations}
\author{Ming-Yuan Chang}

\begin{document}
\maketitle

\begin{abstract}
We study solutions to the Navier-Stokes equations in the class $L^\infty_t C^\alpha_x$. Landau and Lifshitz \cite{LL87} predicted that the Eulerian and Lagrangian temporal structure functions for turbulence exhibit $1/3$ and $1/2$ scaling laws, respectively. These laws were justified for the Euler equations in \cite{Ise23,Ise25}, assuming the spatial structure functions satisfy a $1/3$ scaling law. We demonstrate them in a viscous setting by proving that the $C^\alpha_{t,x}$-norm of the solution and the $C^{1/(1-\alpha)}$-norm of any fluid trajectory can be estimated by the $L^\infty_tC^\alpha_x$-norm independently of the viscosity parameter $\nu>0$, for times bounded away from zero by a positive power of $\nu$.
\end{abstract}

\section{Introduction}

\subsection{Motivation and Statement}
Consider a weak solution $u(t,x):[0,T)\times \T^d\to\R^d$ to the incompressible Navier-Stokes equations:
\begin{equation}\label{Navier-Stokes}\tag{NS}
\begin{dcases}
    \partial_tu+u\cdot\nabla u+\nabla p-\nu \Delta u=0,\\
    \Div u=0,\quad u(0,x)=u_0(x),
\end{dcases}
\end{equation}
where $\nu>0$ and $\Div u_0=0$. 

In the phenomenological theory of turbulence, specifically Kolmogorov's 1941 (K41) theory \cite{Kol41}, the famous $1/3$-law predicts that the absolute spatial structure functions satisfy:
\begin{align}
\langle |u(t,x+\ell)-u(t,x)|^p\rangle^{\frac{1}{p}}\sim \epsilon^{\frac{1}{3}}|\ell|^{\frac{1}{3}}, 
\end{align}
where $\epsilon$ is the energy dissipation rate.
Motivated by this physical prediction, it is natural to consider $L^\infty_t C^\alpha_x$ solutions to \eqref{Navier-Stokes}. In the infinite-Reynolds-number regime, the $C^{1/3}$ regularity turns out to be an important threshold for the validity of energy conservation laws for weak solutions to the Euler equations ($\nu=0$) \cite{CET94,DS12,Ise18}.

On the other hand, Landau and Lifshitz \cite[(33.7) and (33.8)]{LL87} predicted the absolute temporal structure functions from both the Eulerian and Lagrangian perspectives:
\begin{align}
\langle |u(t+\tau,x)-u(t,x)|^p\rangle^{\frac{1}{p}}\sim \epsilon^{\frac{1}{3}}|\tau|^{\frac{1}{3}},\\
\langle |u(t+\tau,X(t+\tau,a))-u(t,X(t,a))|^p\rangle^{\frac{1}{p}}\sim \epsilon^{\frac{1}{2}}|\tau|^{\frac{1}{2}},
\end{align}
where $X(t,a)$ is the Lagrangian flow generated by the vector field $u$. Note that while the temporal scaling in the Eulerian framework is also predicted to follow a $1/3$-law, the regularity is expected to improve to a $1/2$-law in Lagrangian coordinates.

For Euler equations, Isett \cite{Ise23,Ise25} justified mathematically that, if the spatial structure functions are the expected ones, then the temporal structure functions are those predicted. In particular, he demonstrated that for a weak solution $u$ to the Euler equations in the class $L^\infty_tC^\alpha_x$, $0<\alpha<1$, we have the corresponding Hölder regularity $C^\alpha_{t,x}$ in joint space-time. Also, any trajectory $x(t)$ of the solution is of class $C^{\frac{1}{1-\alpha}}$. This implies $u(t,x(t))=x^\prime(t)$ is of class $C^{\frac{\alpha}{1-\alpha}}$, which recovers the predicted exponent when $\alpha=\frac{1}{3}$.

In this paper, we consider the Navier-Stokes equations and prove analogous results in the high-Reynolds-number regime, which reflects physical flows in the presence of small viscosity. Note that the $L^\infty_t C^\alpha_x$ regularity is sufficient to bootstrap the solution to $u\in C^\infty((0,T)\times \T^d;\R^d)\cap C^0([0,T)\times\T^d;\R^d)$. However, the heat regularization generally costs the viscosity $\nu$ to some negative power, causing the estimates to blow up as $\nu\to0$. Thus, it makes sense to ask whether we can obtain the same estimates as in the Euler case \textit{independent of viscosity}.

First, on the Eulerian side, we obtain a joint Hölder bound independent of $\nu$, after the solution evolves for a while:

\begin{thm}\label{Eulerian estimate}
Let $u$ be an $L^\infty_t C^\alpha_x$ solution to \eqref{Navier-Stokes}. Let $\tilde{u}=u-e^{\nu t\Delta}u_0$ be the difference between $u$ and the free heat solution with the same initial data. Then, we have the following estimate:
\begin{align}\label{eq:Eulerian 1}
\|\tilde{u}\|_{L^\infty_x\dot{C}^\alpha_t}\le C\|u\|_{L^\infty_{t,x}}^\alpha\|u\|_{L^\infty_t\dot{C}^\alpha_x}.
\end{align}
In particular, for any $a>0$, we have
\begin{align}\label{eq:Eulerian 2}
\|u\|_{C^\alpha_{t,x}((a\nu,T)\times\T^d)}\le C(\|u\|_{L^\infty_{t,x}}^\alpha\|u\|_{L^\infty_t\dot{C}^\alpha_x}+\|u\|_{L^\infty_tC^\alpha_x}+a^{-\alpha/2}\|u_0\|_{\dot{C}^\alpha_x}).
\end{align}
The constant $C=C_{\alpha,d}$ is independent of $\nu>0$ and $a>0$.
\end{thm}

Note that the pure heat term $||e^{\nu t\Delta}u_0||_{C^\alpha_{t,x}}$ cannot be estimated independently of $\nu$ up to time zero if $u_0\in C^\alpha$ only, while the difference $\tilde{u}$ satisfies such an estimate. This implies that the joint $C^\alpha_{t,x}$ estimate for $u$ up to initial time is exactly obstructed by the heat term.
Since we are mainly concerned with the case where $\nu$ is small, we impose the assumption $t\gtrsim\nu$. Furthermore, we place a free parameter $a>0$ and give a blow-up rate estimate as $t\downarrow 0$.

The main theorem of this paper is the following Hölder estimate for the trajectories. After waiting some time, we can bound the $C^{\frac{1}{1-\alpha}}$-norm of any trajectory by the $L^{\infty}_t\dot{C}^\alpha_x$-norm of $u$ (when $\frac{1}{1-\alpha}\in\Z$ we need to introduce a logarithmic correction):

\begin{thm}\label{Lagrangian estimate}
Let $u$ be an $L^\infty_tC^\alpha_x$ solution to \eqref{Navier-Stokes}, and let $m\ge0$ be the unique integer such that $m<\frac{1}{1-\alpha}\le m+1$. Then, for any $a>0$ and any $t_1,t_2\ge a\|u\|_{L^\infty_t\dot{C}^\alpha_x}^{\frac{-2}{1+\alpha}}\nu^{\frac{1-\alpha}{1+\alpha}}$, the following estimates hold:
\begin{enumerate}[label=(\roman*), font=\upshape]
    \item If $\frac{1}{1-\alpha}\not\in\Z$ and $\frac{1}{1-\alpha}=m+\beta$,
    \begin{align}
    |x^{(m)}(t_1)-x^{(m)}(t_2)|\le C(a^{-m}+1)\|u\|^{\frac{1}{1-\alpha}}_{L^\infty_t\dot{C}^\alpha_x}|t_1-t_2|^{\beta}.
    \end{align}
    \item If $\frac{1}{1-\alpha}\in\Z$ and $\frac{1}{1-\alpha}=m+1$,
    \begin{align}
    |x^{(m)}(t_1)-x^{(m)}(t_2)|\le C(a^{-m}+1)\|u\|^{\frac{1}{1-\alpha}}_{L^\infty_t\dot{C}^\alpha_x}|t_1-t_2|(1-\log^-(\|u\|_{L^\infty_t\dot{C}^\alpha_x}|t_1-t_2|)),
    \end{align}
    where $\log^-(t)=\min\{\log(t),0\}$.
\end{enumerate}
The constants $C=C_{\alpha,d}$ are independent of $\nu>0$ and $a>0$.

\end{thm}

The theorem is easier to prove in the regime $0<\alpha\le \frac{1}{2}$, which contains the case $\alpha=\frac{1}{3}$. See Theorem \ref{Lagrangian estimate less than one half} for the simplified statement.

To provide some heuristics, let $\ell$ be a spatial scale and $\tau$ be a time scale. $\tau_d=\nu^{-1}\ell^{2}$ is the time scale at which the diffusion starts to affect the scale $\ell$ significantly. The eddy turnover time, $\tau_e=\|u\|_{L^\infty_t\dot{C}^\alpha_x}^{-1}\ell^{1-\alpha}$, is the time scale at which the energy starts to cascade significantly to small scales. When $\tau_{e}\ll \tau_d$, diffusion is negligible and the estimates are the same as those for the Euler equations. When $\tau_e\gtrsim \tau_d$, we reach the \textit{dissipation scales} $\ell\lesssim \|u\|_{L^\infty_t\dot{C}^\alpha_x}^{\frac{-1}{1+\alpha}}\nu^{\frac{1}{1+\alpha}}$, and the two effects balance at times $\tau\gtrsim \|u\|_{L^\infty_t\dot{C}^\alpha_x}^{\frac{-2}{1+\alpha}}\nu^{\frac{1-\alpha}{1+\alpha}}$, which is exactly the scale stated in Theorem \ref{Lagrangian estimate}.

Finally, we briefly mention some related results for other hydrodynamic equations. Colombo and De Rosa \cite{CDR20} established the joint $C^\alpha_{t,x}$ estimate for the hypodissipative Navier-Stokes involving $(-\Delta)^{\gamma}$, $0<\gamma<\frac{1}{2}$. Similarly, Wang, Mei, and Liu \cite{WML23} proved an analogous result for the Surface Quasi-Geostrophic equation. We also refer to \cite{CVW15} for other improved regularity results for the Euler equations in a Lagrangian setting.

\subsection{Strategy for the proof}
Let us briefly explain the strategy of the proof. The proof is based on the estimates for the Euler equations established in \cite{Ise23,Ise25}. Compared to the Euler equations, we have an additional diffusion term $\nu\Delta u$, which needs to be taken into consideration at dissipation scales. 

However, applying directly the Hölder assumption, we can only estimate the Littlewood-Paley piece by $\|\nu\Delta P_ku\|_{L^\infty}\lesssim \nu 2^{(2-\alpha)k}\|u\|_{\dot{C}^\alpha}$. On the other hand, $\|P_{\le k}u\cdot\nabla P_ku\|_{L^\infty}\lesssim 2^{(1-\alpha)k}\|u\|_{\dot{C}^\alpha}^2$. Thus, for a fixed $\nu>0$ and sufficiently large $k$, it seems that the term $\nu\Delta P_ku$ has larger magnitude over all the terms appearing in the estimates for the Euler equations. To obtain estimates that align with the non-viscous case, one must utilize the parabolic structure of the equation. 

The Eulerian estimate is easier. From Duhamel's formula (Lemma \ref{Parabolic estimate}), one can bound it by the same bound for the drift term $P_{\le k}u\cdot\nabla P_ku$ above, which suffices for the Eulerian $C^\alpha_{t,x}$ estimate.

The improved regularity of trajectories relies on absorbing the large transport term into the material derivative to obtain a better estimate, so Duhamel's estimate is insufficient in this case. Instead, we observe that the transport term itself tends to reorganize the solution without increasing the norm, suggesting that the dissipative effect on the norms remains largely unhindered. Therefore, we allow the time evolution to damp the high frequency modes, so that the term $\nu\Delta u$ becomes comparable to other terms. 

However, it is not immediately obvious how to exploit the diffusion effect in the $L^\infty$ norm. For the standard heat equation, we can estimate $\|e^{t\Delta}P_kf\|_{L^\infty}\le Ce^{-c2^{2k}t}\|P_kf\|_{L^\infty}$, where $P_k$ denotes the usual dyadic Littlewood-Paley projection. On the other hand, the maximum principle for a transport-diffusion equation only guarantees that the $L^\infty$ norm is non-increasing in general. It turns out that, in dimension $2$, this problem can be solved by simply replacing the dyadic decomposition by a $(1+\delta)$-adic decomposition, for $\delta>0$ small enough. Heuristically, bounded functions with Fourier support in a thin annulus behave like those with support on a sphere, which, in dimension 2, are eigenfunctions of the Laplacian. The following statement captures this intuition, acting as an illustrative proxy for the precise technical lemma (Lemma \ref{Thin annulus lemma}) used later.

\begin{lemma}\label{Thin annulus decay}
Let $d=1,2$

For any $\epsilon>0$, there exists a $\delta=\delta(\epsilon,d)>0$ depending only on $\epsilon$ and $d$ such that the following holds:

For any $R>0$ and any $f\in L^p(\T^d;\R^m)$ or $L^p(\R^d;\R^m)$, $1\le p\le\infty$ with $\spt\hat{f}\subset\{\xi:R(1+\delta)^{-1}<|\xi|< R(1+\delta)\}$, we have the following bound for the solution to the free heat equation $e^{t\Delta}f$:
\begin{align}\label{eq:thin annulus decay}
\|e^{t\Delta}f\|_{L^p}\le e^{-(1-\epsilon)R^2t}\|f\|_{L^p}.
\end{align}
\end{lemma}

Note that we do not have a constant $C$ in front of the estimate, and this is equivalent to an instantaneous dissipation on the $L^\infty$ norm of the $(1+\delta)$-adic pieces, which is what we utilize to obtain a damping on the $L^\infty$ norm via the maximum principle.

In dimension $\ge 3$, this statement is no longer true, and we provide a substitute lemma (Lemma \ref{lem:generalized thin annulus lemma}) that works well with our problem.

We also present our first argument for the theorem in Section \ref{sec:alternative approach}, which relies on $L^p$-energy estimates. Since $\Div u=0$, the transport term only redistributes the $L^p$-norm without increasing it, allowing us to obtain the $L^\infty$ estimate as a limit of $L^p$ estimates.

\subsection{Structure of the Paper}

In Section \ref{sec:Eulerian estimate}, we will introduce the $(1+\delta)$-adic Littlewood-Paley projections, and prove the Eulerian estimate (Theorem \ref{Eulerian estimate}). The proof depends on several parabolic estimates, whose proofs are given in the Appendix. The freedom for varying the base number $(1+\delta)$ is not important for the Eulerian estimate, but it will become important in the Lagrangian estimate.

In Section \ref{sec:thin annulus lemma}, we will prove the key technical lemma for the proof in dimension $2$: the \textit{thin annulus lemma}. It demonstrates the instataneous dissipation of bounded functions with Fourier support in a thin annulus, and we will apply it to the $(1+\delta)$-adic projections of the solution. Strictly speaking, the Lagrangian estimate only requires Lemma \ref{Thin annulus lemma}, but we will complete the whole story with the proof of Theorem \ref{Thin annulus decay}. Then, we will provide a generalization of this lemma (Lemma \ref{lem:generalized thin annulus lemma}) to tackle the case $d\ge 3$.

In Section \ref{sec:Lagrangian estimate}, we will prove the Lagrangian estimate in the range of $0<\alpha\le\frac{1}{2}$ (Theorem \ref{Lagrangian estimate less than one half}), which contains the critical $C^{\frac{1}{3}}$ case. The proof closely follows the scheme of \cite{Ise23,Ise25} using Littlewood-Paley projections, with an additional maximum principle. We will prove the result first in dimension 2, and then in dimension 3. It will then become clear why we want to consider Littlewood-Paley projections to thin annuli as in the previous section. The proof for the case $\alpha>\frac{1}{2}$ adds technical complications to the argument and is left to the next section, while the most relevant exponent $\frac{1}{3}$ stays in the range where the proof is easier.

In Section \ref{sec:higher derivatives}, we will set up an induction scheme of estimates (Theorem \ref{Main lemma}) needed for the part $\frac{1}{2}<\alpha<1$ left in Theorem \ref{Lagrangian estimate} on estimating higher material derivatives. We will finish the entire proof in the final Section \ref{sec:proof of main lemma}. The main argument is the same as in the previous case, but there are a few new technical estimates, including commuting material derivatives with convolution operators. The main lemma in \cite{Ise23} is replaced by Theorem \ref{Main lemma}, and we provide some explicit commutator formulas that accommodate general situations. In particular, it will be clear that the proof for the Navier-Stokes equations is parallel to the proof for the Euler case.

\section{Eulerian Estimate}\label{sec:Eulerian estimate}
In this section, we prove Theorem \ref{Eulerian estimate} on the joint $C^\alpha_{t,x}$ bound, which we call the Eulerian estimate.

\subsection{Littlewood-Paley projections}

First, we briefly introduce the $(1+\delta)$-adic Littlewood-Paley decomposition on torus, which is defined similarly as the usual dyadic ones, with some constant bounds depending on $\delta$. 

We define the $(1+\delta)$-adic Littlewood-Paley projections $P_k$ on $\T^d$: Take a smooth cut-off function $m_0(\xi)\in C^\infty(\R^d)$ with $\spt m_0\subset\{(1+\delta)^{-1}< |\xi|< (1+\delta)\}$ and $m_0(-\xi)=m_0(\xi)$, satisfying
\[\sum_{k\in\Z}m_k(\xi)=1,\]
where $m_k(\xi)=m_0(\frac{\xi}{(1+\delta)^k})$.
For a function $f:\T^d\to\R$, we view it as a $(2\pi\Z)^d$-periodic function on $\R^d$ and define
\begin{align*}
P_kf(x)=\int_{\R^d}\varphi_k(y)f(x-y)\,dy=\int_{\T^d}(\sum_{n\in\Z^d}\varphi_k(y+n))f(x-y)\,dy,
\end{align*}
where $\hat{\varphi}_k=m_k$.

Denote $P_{\le k}f=\sum_{l\le k}P_lf+|\T^d|^{-1}\int_{\T^d}f$. We also include the zero mode that is undetected by the operator $P_k$.
Also define $P_{>k}=Id-P_{\le k}$ and $P_{[k_1,k_2]}=\sum_{k_1\le k\le k_2}P_k$. 

Note that $\spt\widehat{P_kf}\subset\{\xi\in\Z^d:(1+\delta)^{k-1}<|\xi|< (1+\delta)^{k+1}\}$, so that derivatives cost roughly a factor of $(1+\delta)^{k}$. To be precise, placing derivatives on the convolution kernel $\sum_{k-2\le k^\prime\le k+2}\varphi_{k^\prime}$ yields the following estimates:
\[\|\nabla^cP_kf\|_{L^\infty}=\|\nabla^cP_{[k-2,k+2]}P_kf\|_{L^\infty}\lesssim_{c,\delta} (1+\delta)^{ck}\|P_kf\|_{L^\infty},\]
for all $c\ge0$. The symbol $A\lesssim B$ means $A\le CB$ with a constant $C$ independent of $k,f$.

The Littlewood-Paley projections are particularly useful for handling Hölder estimates. For example, we have the following norm equivalence:
\begin{align}\label{eq:Hölder characterization}
\|f\|_{\dot{C}^\alpha}\sim \sup_{k\in\Z}(1+\delta)^{\alpha k}\|P_kf\|_{L^\infty},
\end{align}
where $0<\alpha<1$. In particular, we obtain the estimate $\|P_ku\|_{L^\infty}\lesssim_{\delta} (1+\delta)^{-\alpha k}\|u\|_{\dot{C}^\alpha}$. Furthermore, we can bound
\[\|\nabla P_{\le k}u\|_{L^\infty}\le\sum_{l\le k}\|\nabla P_lu\|_{L^\infty}\lesssim\sum_{l\le k}(1+\delta)^{(1-\alpha)l}\|u\|_{\dot{C}^\alpha}\lesssim (1+\delta)^{(1-\alpha)k}\|u\|_{\dot{C}^\alpha}.\]

Another useful observation is the product formula:
\begin{align}
P_k(fg)=\sum_{k^\prime,k^{\prime\prime}}P_k(P_{k^\prime}fP_{k^{\prime\prime}}g)=\sum_{(k^\prime,k^{\prime\prime})\in HH\cup HL\cup LH}P_k(P_{k^\prime}fP_{k^{\prime\prime}}g),
\end{align}
where $HH=\{(k^\prime,k^{\prime\prime}):k^\prime,k^{\prime\prime}\ge k, |k^\prime-k^{\prime\prime}|\le 2+\frac{\log 2}{\log(1+\delta)}\}$, $HL=\{(k^\prime,k^{\prime\prime}):|k^\prime-k|\le 2+\frac{\log 2}{\log(1+\delta)},k^{\prime\prime}\le k\}$, and $LH=\{(k^\prime,k^{\prime\prime}):|k^{\prime\prime}-k|\le 2+\frac{\log 2}{\log(1+\delta)},k^{\prime}\le k\}$.

\subsection{The Proof}

Now, we can start to derive the Eulerian estimate. In fact, we consider a slightly general class of equations, as the Eulerian estimate only depends on this clear structure:
\begin{align}
\partial_t u-\nu\Delta u=T\nabla(u\otimes u),
\end{align}
where $T:\mathscr{S}(\R^n;(\R^{n})^3)\to\mathscr{S}(\R^n;\R^n)$ is a Fourier (matrix) multiplier satisfying the bound $\|P_kT\|_{L^\infty\to L^\infty}\le C$ independent of $k$. We take $P_k$ to be the usual dyadic Littlewood-Paley projection ($1+\delta=2$ in the definition), as we don't need $\delta$ small in this part. In particular, Hörmander-Mikhlin type condition for $T$ ensures this bound for $P_kT$. In the case of Navier-Stokes equations, the operator is defined by
\begin{align}
T^i(A_{c}^{ab})=-A^{ai}_a-(-\Delta)^{-1}\partial_a\partial_bA^{ab}_i=-\partial_a(u^au^i)-(-\Delta)^{-1}\partial_a\partial_b\partial_i(u^au^b),
\end{align}
where $(\nabla(u\otimes u))^{ab}_c=\partial_c(u^au^b)$.

We can estimate the following term:
\[\|P_kT\nabla(u\otimes u)\|_{L^\infty}=\|P_{[k-2,k+2]}TP_k\nabla(u\otimes u)\|_{L^\infty}\lesssim 2^k\|P_k(u\otimes u)\|_{L^\infty}.\]
From the frequency interactions in the product, we obtain
\[P_k(u\otimes u)=P_k(P_{\ge k-3}u\otimes u+P_{<k-3}u\otimes P_{[k-3,k+3]}u).\]
In each term, at least one factor is supported at high frequency, which yields the estimate
\begin{align}
\|P_kT\nabla(u\otimes u)\|_{L^\infty}\lesssim 2^{(1-\alpha)k}\|u\|_{L^\infty}\|u\|_{\dot{C}^\alpha}.
\end{align}

Let $\tilde{u}=u-e^{\nu t\Delta}u_0$, which solves the equation:
\begin{equation}
\begin{dcases}
    (\partial_t-\nu \Delta)\tilde{u}=T\nabla(u\otimes u),\\
    \tilde{u}(0,x)=0.
\end{dcases}
\end{equation}
By a parabolic estimate (Lemma \ref{Parabolic estimate}), we obtain
\begin{align}
\|\partial_tP_k\tilde{u}\|_{L^\infty_{t,x}}\lesssim\|P_kT\nabla(u\otimes u)\|_{L^\infty_{t,x}}\lesssim2^{(1-\alpha)k}\|u\|_{L^\infty_{t,x}}\|u\|_{L^\infty_t\dot{C}^\alpha_{x}}.
\end{align}
We can now finish the proof of Theorem \ref{Eulerian estimate}. For any $t\ge0$ and $h>0$, we have
\begin{equation}
\begin{split}
|\tilde{u}(t+h,x)-\tilde{u}(t,x)|
&\le\sum_{k=-\infty}^{\infty}|P_k\tilde{u}(t+h,x)-P_k\tilde{u}(t,x)|\\
&\le\sum_{k=-\infty}^{K}\|\partial_tP_k\tilde{u}\|_{L^\infty_{t,x}}|h|+2\sum_{k=K}^{\infty}\|P_k\tilde{u}\|_{L^\infty_{t,x}}\\
&\lesssim\sum_{k=-\infty}^{K}2^{(1-\alpha)k}\|u\|_{L^\infty_{t,x}}\|u\|_{L^\infty_t\dot{C}^\alpha_x}|h|+\sum_{k=K}^{\infty}2^{-\alpha k}\|u\|_{L^\infty_{t}\dot{C}^\alpha_x}\\
&\lesssim (2^{(1-\alpha)K}\|u\|_{L^\infty_{t,x}}|h|+2^{-\alpha K})\|u\|_{L^\infty_{t}\dot{C}^\alpha_x}.
\end{split}
\end{equation}
Take $2^K\sim (\|u\|_{L^\infty_{t,x}}|h|)^{-1}$ (one of the closest dyadic numbers) to balance the two terms, we finally get 
\begin{align}
|\tilde{u}(t+h,x)-\tilde{u}(t,x)|\le C|h|^\alpha\|u\|_{L^\infty_{t,x}}^\alpha\|u\|_{L^\infty_t\dot{C}^\alpha_x}.  
\end{align}
This completes the proof of the first estimate \eqref{eq:Eulerian 1}. For the second estimate \eqref{eq:Eulerian 2}, we use some parabolic estimates (Lemma \ref{Parabolic estimate 2}): for a fixed $t\ge0$ and $h>0$,
\begin{align}
\|e^{\nu (t+h)\Delta}u_0-e^{\nu t\Delta}u_0\|_{L^\infty_x}\lesssim\|e^{\nu t\Delta}u_0\|_{\dot{C}^{2\alpha}_{*}}\cdot(\nu|h|)^\alpha\lesssim\|u_0\|_{\dot{C}^\alpha_x}\cdot(\nu t)^{-\alpha/2}\cdot(\nu|h|)^\alpha,
\end{align}
where the Hölder-Zygumnd space $\dot{C}^{\beta}_*$ is defined by
\[\|f\|_{\dot{C}^\beta_{*}}:=\sup_{k\in\Z}2^{\beta k}\|P_kf\|_{L^\infty}\]
for all $\beta>0$. In particular, this coincides with the usual Hölder norm when $\beta\not\in\N$.

We conclude the estimate
\begin{align}
\|e^{\nu t\Delta}u_0\|_{L^\infty_x\dot{C}^\alpha_t((a\nu,T)\times \T^d)}\le C\nu^{\alpha/2}(a\nu)^{-\alpha/2}\|u_0\|_{\dot{C}^\alpha}=Ca^{-\alpha/2}\|u_0\|_{\dot{C}^\alpha}.
\end{align}

\section{Thin Annulus Lemma}\label{sec:thin annulus lemma}

In this section, we establish estimates for bounded functions whose Fourier support is localized to a thin annulus.
In particular, we are interested in obtaining a positive lower bound of $-\Delta f\cdot f$ at the points where $|f|$ achieves its supremum, which helps us make use of the dissipative structure of the equation. The key idea here is that a function with frequency support in a thin annulus should, in a certain sense, be close to one with support on a sphere.

In dimension $2$, the functions look like an eigenfunction of the Laplacian. The thin annulus lemma (Lemma \ref{Thin annulus lemma}), which demonstrates a $\Delta$-eigenfunction behavior of such functions at maximum points, will lead us to a ``quantitative'' strong maximum principle in the proof of Theorem \ref{Lagrangian estimate}. 

For dimension $d\ge3$, the lemma is no longer true, but there is a replacement (Lemma \ref{lem:generalized thin annulus lemma}), which works well with Navier-Stokes equations. 

The $L^p$ analogue (Theorem \ref{Thin annulus decay} and \ref{Instantaneous dissipation rate}) of instantaneous dissipation of the heat equation completes the heuristic and helps us set up another approach (Section \ref{sec:alternative approach}).

\subsection{Dimension 1 and 2}

When $d=1,2$, any $f\in L^\infty(\R^d)$ with $\spt\hat{f}\subset S^{d-1}$ satisfies $(-\Delta-1)f=0$. We present a proof based on kernel analysis, a technique we will employ several times throughout the paper. Consider the Fourier multiplier
\begin{align}
\widehat{T_\delta f}=(|\xi|^2-1)\varphi(\frac{|\xi|-1}{\delta})\hat{f},
\end{align}
where $\varphi\in C^\infty_c(\R)$ is a cut-off with $\varphi\equiv 1$ near $0$. We claim the following bound:
\begin{equation}
\|T_\delta\|_{L^\infty\to L^\infty}\lesssim
\begin{dcases}
    \delta &, d=1,\\
    \delta^{1/2} &, d=2.
\end{dcases}
\end{equation}
Since $(-\Delta-1)f=T_{\delta}f$ for $\spt\hat{f}\subset S^{d-1}$, the statement follows directly from the claim by taking $\delta\to0$. To see the bound, let the kernel of this operator be $K_\delta$, we can estimate
\begin{align}
\|x^\alpha K_\delta\|_{L^2}=\|\partial_\xi^\alpha \big((|\xi|^2-1)\varphi(\frac{|\xi|-1}{\delta})\big)\|_{L^2}\lesssim \delta^{3/2-|\alpha|}
\end{align}
for any multiindex $\alpha$ and $\delta\le 1$. We thus obtain
\begin{align}\label{eq:multiplier calculation}
\|K_\delta\|_{L^1}\le \|(1+|\delta x|^2)^lK_\delta\|_{L^2}\cdot\|(1+|\delta x|^2)^{-l}\|_{L^2}\lesssim \delta^{(3-d)/2}
\end{align}
by picking any fixed $l>\frac{d}{4}$.
From this proof, it turns out that a perturbative statement holds true. For every $\epsilon>0$, there exists a $\delta>0$ such that for every $\spt \hat{f}\subset B_\delta(S^{d-1})$, 
\[\|(-\Delta-1)f\|_{L^\infty}\le \epsilon\|f\|_{L^\infty}.\]
In particular, if $x$ is a point such that $|f|(x)=\|f\|_{L^\infty}$, then
\[-\Delta f\cdot f(x)\ge |f|^2(x)-\epsilon\|f\|_{L^\infty}^2=(1-\epsilon)|f|^2(x).\]
From the rescaling $f(\frac{x}{R})$, we conclude the following thin annulus lemma:

\begin{lemma}\label{Thin annulus lemma}
Let $d=1,2$.

For any $\epsilon>0$, there exists a $\delta>0$, depending only on $\epsilon,d$ such that the following statement holds:

Given any $R>0$, any bounded function $f:\R^d\to\R^m$ whose Fourier transform is supported in the thin annulus: $\spt\hat{f}\subset\{\xi:R(1+\delta)^{-1}<|\xi|< R(1+\delta)\}$, and any point $x\in \R^d$ attaining the sup: $|f(x)|=\|f\|_{L^\infty}$, we have the lower bound
\begin{align}
-\Delta f(x)\cdot f(x)\ge (1-\epsilon)R^2|f|^2(x).
\end{align}
\end{lemma}

For the purpose of our proof in 2d, we will fix a $\delta$ for $\epsilon=\frac{1}{2}$.

We present the proof of equivalence between Lemma \ref{Thin annulus lemma} and the $L^\infty$ case of Theorem \ref{Thin annulus decay}. It will provide some important insights for the generalization in higher dimensions.

Note that Lemma \ref{Thin annulus lemma} is actually the instantaneous dissipation rate of $L^\infty$-norm for heat equation. To see this, we invoke the well-known maximum principle trick, also known as Hamilton's trick (\cite{Ham86}, Lemma 3.5).

First, we convolute $\hat{f}$ with a mollifier and assume $f$ vanishes at infinity. It suffices to prove Theorem \ref{Thin annulus decay} under this assumption. Let $f(t,x)=e^{t\Delta}f(x)$. The maximum principle says that $\|f(t,\cdot)\|_{C^0}$ is a locally Lipschitz function in time, and for a.e. $t$ where it is differentiable, we have
\begin{align}
\frac{d}{dt}\frac{1}{2}\|f(t,\cdot)\|_{C^0}^2= \frac{\partial f}{\partial t}(t,x)\cdot f(t,x)\bigg|_{x:|f|(t,x)=\|f(t,\cdot)\|_{C^0}}
\end{align}
Since $f$ vanishes at infinity, for each time $t$, there exists point $x\in\R^d$ with $|f|(t,x)=\|f(t,\cdot)\|_{C^0}$. Assume $f\neq 0$, and say $x\in\R^d$ is a point such that $|f|(t,x)=\|f(t,\cdot)\|_{C^0}>0$, then by Lemma \ref{Thin annulus lemma},
\begin{align*}
\frac{\partial f}{\partial t}(t,x)\cdot f(t,x)=\Delta f(t,x)\cdot f(t,x)\le -(1-\epsilon)R^2|f|^2(t,x)=-(1-\epsilon)R^2\|f(t,\cdot)\|_{C^0}^2.
\end{align*}
This yields
\begin{align}
\frac{d}{dt}\|e^{t\Delta}f\|_{C^0}\le -(1-\epsilon)R^2\|e^{t\Delta}f\|_{C^0},
\end{align}
for a.e. $t$, and thus conclude theorem \ref{Thin annulus decay}. Conversely, for a function $f$ in $L^\infty(\R^d;\R^m)$ or $L^\infty(\T^d;\R^m)$ satisfying the assumptions and a point $x$ with $|f|(x)=\|f\|_{L^\infty}(t)$,
\begin{align}
\frac{|e^{t\Delta}f|^2(x)-|f|^2(x)}{t}\le\frac{e^{-2(1-\epsilon)R^2t}-1}{t}|f|^2(x).
\end{align}
Taking $t\to0^+$ yields Lemma \ref{Thin annulus lemma}.

A final observation is that, bounding $\|K_\delta\|_{L^1}$ also bounds the operator norm $\|T_{\delta}\|_{L^p\to L^p}$ uniformly. Thus, we can bound for $\spt\hat{f}\subset B_{\delta}(S^{d-1})$:
\[|\int(-\Delta-1)f\cdot f|f|^{p-2}|\le \epsilon \|f\|_{L^p}^p.\]
This yields the following theorem on instantaneous dissipation rate, which again is equivalent to Theorem \ref{Thin annulus decay} for the cases $p<\infty$:

\begin{lemma}\label{Instantaneous dissipation rate}
Let $d=1,2$.

For any $\epsilon>0$, there exists a $\delta>0$, depending only on $\epsilon,d$ such that the following statement holds:

Given any $R>0$, and any function $f$ in $L^p(\R^d;\R^m)$ or $L^p(\T^d;\R^m)$ with Fourier transform supported in the thin annulus: $\spt\hat{f}\subset\{\xi:R(1+\delta)^{-1}\le|\xi|\le R(1+\delta)\}$, we have the lower bound

\begin{align}
\int -\Delta f\cdot f|f|^{p-2}\ge (1-\epsilon)R^2\int|f|^p.
\end{align}
\end{lemma}

We are going to use this $p$-th norm estimate in the second proof of Theorem \ref{Lagrangian estimate}.

\subsection{Replacement in \texorpdfstring{$d\ge 3$}{d3}}\label{sec:generalized thin annulus lemma}
If we try to prove the same result in $d\ge 3$, the proof breaks down. Our bound \eqref{eq:multiplier calculation} for the operator $T_{\delta}$ does not shrink as we shrink $\delta>0$. It turns out that this is a real obstruction, since the Fourier transform of a $L^\infty$ functions can contain derivatives of measures. For example, in 3d, we have
\begin{align}
\widehat{\partial_r\,d\sigma}(x)=\sqrt{\frac{2}{\pi}}(-\cos|x|+\frac{\sin|x|}{|x|})=:f\in L^\infty(\R^3),
\end{align}
where $\,d\sigma$ is the sphere measure on $S^2$. $f$ is not a eigenfunction of the laplacian. Instead, it satisfies $(-\Delta-1)^2f=0$. 

More generally, by analyzing the multiplier $(|\xi|^2-1)^l\varphi(\frac{|\xi|-1}{\delta})$ similarly as in \eqref{eq:multiplier calculation}, we obtain the following statement:
\begin{lemma}\label{lem:iterated laplacian}
Let $j_0(d)=\lfloor{\frac{d-1}{2}}\rfloor$.
For every $f\in L^\infty(\R^d)$ with $\spt\hat{f}\subset S^{d-1}$, we know that
\begin{align}
(-\Delta-1)^{j_0(d)+1}f=0.
\end{align}
\end{lemma}

Similarly, there exists counterexamples for Lemma \ref{Thin annulus lemma} in 3d. The following one is suggested by Google Gemini. If $f=3\frac{\sin|x|}{|x|}-\cos|x|$, then $\hat{f}=\sqrt{\frac{\pi}{2}}(\partial_r\,d\sigma+2\,d\sigma)$, which is supported in $S^2$. On the other hand, $f$ achieves maximum at $x=0$, with $\Delta f(0)=0$. Thus, we cannot hope for a positive lower bound estimate for $-\Delta f\cdot f$ even if we shrink the annulus.

However, there are still hopes to make use of dissipation in the flow of time. Recall that for heat equation $f(t)=e^{t\Delta}f$, we have
\[\|f(t)\|_{L^\infty}\le Ce^{-ct}\|f\|_{L^\infty}.\]

Writing $Ce^{-ct}=e^{-c(t-t_0)}$ for some $t_0>0$ allows us to interpret this in the context of the heat equation: outside an interval of length $t_0$, the inequality $-\Delta f(t,x)\cdot f(t,x)\ge c|f|^2(t,x)$ holds on time average at points $x$ of supremum.

This interpretation suggests that these fine-tuned examples do not occur frequently in the flow of a heat solution. Lets see locally in time what happens when we evolve the counterexample.
\[e^{-t|\xi|^2}(\partial_r\,d\sigma+2\,d\sigma)=e^{-t}(\partial_r\,d\sigma+(2+2t)\,d\sigma).\]
The corresponding function in physical space is $C(t)((3+2t)\frac{\sin|x|}{|x|}-\cos|x|)$. The cancellation is spoiled, and the ratio $\frac{-\Delta f\cdot f}{|f|^2}(0)=\frac{t}{1+t}$ grows from $0$ nondegenerately for small $t>0$ ($\frac{d}{dt}|_{t=0}\frac{t}{1+t}\neq 0$).

Of course, the situation we want to apply is not heat, but a diffusion transport equation with forcing term. It is not obvious how we could prevent the drift and the forcing from repositioning the solution into these delicate balance. Instead, the lesson we took from the heat equation is that a Fourier multiplier could possibly perturb the degeneracy. We end this section with a rigorous statement following this idea.

In 3d, if $\spt\hat{f}\subset B_\delta(S^2)$ and $\delta=\delta(\epsilon)$ is small enough, we can prove that either 
\[\|(-\Delta-1)f\|_{L^\infty}\le \epsilon\|f\|_{L^\infty}\]
holds, or the "perturbed" function $g:=(-\Delta-1)f$ satisfies
\[\|(-\Delta-1)g\|_{L^\infty}\le\epsilon\|g\|_{L^\infty}.\]
Heuristically, either $f$ already looks like an eigenfunction of the Laplacian, or it contains a part that looks like a function in $\ker (-\Delta-1)^2$. In the latter case, the function $g$ looks like an eigenfunction of the Laplacian. In fact, we can replace the above $g$ by $\tilde{g}=(-\Delta-\lambda^2)f$ if $\lambda$ is sufficiently close to $1$ depending on $\delta$. This gives us an abundance of auxiliary functions that dissipate well when the original function $f$ does not, which will be the key to the proof in 3d and above. We state a general lemma and use proof by contradiction to demonstrate the idea. We will also present a direct and quantitative proof during the proof for 3d (see \eqref{eq:concrete proof} and nearby arguments).

\begin{lemma}\label{lem:generalized thin annulus lemma}
Given a modulus $\phi:[0,1]\to\R_{\ge 0}$ with $\lim_{\delta\to 0}\phi(\delta)=0$ and a $\epsilon>0$, there exists a $\delta_0=\delta_0(\phi,\epsilon,d,m)$ such that the following statement is true for all $0<\delta\le \delta_0$:

For all $f_0\in L^\infty(\R^d;\R^m)$ with $\spt\hat{f}_0\subset\{\xi:R(1+\delta)^{-1}<|\xi|< R(1+\delta)\}$, if we let $f_j=(-\Delta-\lambda_j^2)f_{j-1}$, where $\lambda_j= R(1+\gamma_j)$ with any $|\gamma_j|\le\phi(\delta)$, $j=1,\dots,j_0(d)$, then there exists a $0\le j\le j_0(d)$ such that
\begin{align}
\|(-\Delta-R^2)f_j\|_{L^\infty}\le \epsilon R^2\|f_j\|_{L^\infty}.
\end{align}
\end{lemma}
\begin{proof}
We may assume $R=1$ by scaling.

Assume that we have a sequence $f_j^{(n)}$ contradicting the statement for some $\delta_n\le\frac{1}{n}$. There are two cases.

\textit{Case 1.} There exists $1\le j\le j_0(d)$ such that $\limsup_{n\to \infty}\frac{\|f_{j-1}^{(n)}\|_{L^\infty}}{\|f_{j}^{(n)}\|_{L^\infty}}=\infty$. 

In this case, we normalize $\|f_{j-1}^{(n)}\|_{L^\infty}=1$. Now $f_{j}^{(n)}$ converges to $0$ strongly in $L^\infty$, while on the other hand,
\[(-\Delta-1)f_{j-1}^{(n)}=f_{j}^{(n)}+(2\gamma_j^{(n)}+(\gamma_j^{(n)})^2)f_{j-1}^{(n)}.\]
This also converges to zero strongly, thus contradicting $\|(-\Delta-1)f_{j-1}\|_{L^\infty}>\epsilon$.

\textit{Case 2.} For each $j,n$, $\|f_{j-1}^{(n)}\|_{L^\infty}\le C\|f_{j}^{(n)}\|_{L^\infty}$. In particular, if we normalize $\|f_{j_0(d)}^{(n)}\|_{L^\infty}=1$, then every sequence $f_{j}^{(n)}$ is bounded in $L^\infty$ norm. Translate so that $|(-\Delta-1)f_{j_0(d)}^{(n)}|(0)>\epsilon$, and pick a subsequence so that for every $j$, $f_{j}^{(n)}$ converges $L^\infty$ weakly* and $C^\infty_{loc}$ to $f_j^{(\infty)}$. It is easy to see that $(-\Delta-1)f_{j-1}^{(\infty)}=f_j^{(\infty)}$ by the control $|\gamma_j^{(n)}|\to0$. Thus,
\[f_{j_0(d)}^{(\infty)}=(-\Delta-1)^{j_0(d)}f_0^{(\infty)}.\]
Note that $f_0^{(\infty)}$ has frequency supported in $S^{d-1}$. By Lemma \ref{lem:iterated laplacian}, we know $(-\Delta-1)f_{j_0(d)}^{(\infty)}=0$, contradicting the conclusion $|(-\Delta-1)f_{j_0(d)}^{(\infty)}|(0)\ge\epsilon$.

\end{proof}

\section{Lagrangian estimates in the regime \texorpdfstring{$1<\frac{1}{1-\alpha}\le 2$}{1 < 1/(1-alpha) le 2}}\label{sec:Lagrangian estimate}

In this section, we prove the Lagrangian trajectory estimates (Theorem \ref{Lagrangian estimate}) in the case of $0<\alpha\le \frac{1}{2}$, which contains the critical exponent $\alpha=\frac{1}{3}$ in the turbulence theory. We will only present the proof in dimension 2 and 3 to make the ideas clear.

We will first observe that the key lies in an estimate of the material derivative $(\partial_t+P_{\le k}u\cdot\nabla)P_{\le k}u$. Then, after evoking some well-known estimates for the commutators and the pressure, we further reduce the problem to the estimate of the diffusion term $\nu\Delta P_ku$. Evolving the flow and damping the $L^\infty$ norm using the thin annulus lemma, we can make the dissipation term as good as other terms at sufficiently small scales after a certain amount of time and thus conclude the theorem in dimension $2$.

The key to higher dimensions lies in finding auxiliary functions and making use of the dissipation in either one of the functions. The main argument is the same in any dimension, and we will present the extra ingredients in dimension 3 after the proof in dimension 2, which is easier and more intuitive.

\subsection{Main Proof in dimension 2}

The proof starts with the idea in \cite{Ise23} and \cite{Ise25}, which involves comparing the trajectories with those of the coarse scale flow. For a time scale $\tau>0$, we associate it with a spatial scale $(1+\delta)^{-k}$ given by the eddy turnover time $\tau\sim\|u\|^{-1}_{L^\infty_t\dot{C}^\alpha_x}(1+\delta)^{(\alpha-1)k}$.
Considering the coarse scale flow $\partial_t x_{(k)}=P_{\le k}u(t,x_{(k)}(t))$ with $x_{(k)}(t_0)=x(t_0)$, we can decompose the difference as follows:
\begin{align}\label{eq:difference decomposition}
\begin{split}
&u(t_0+\tau,x(t_0+\tau))-u(t_0,x(t_0))\\
&=\underbrace{(u(t_0+\tau,x(t_0+\tau))-P_{\le k}u(t_0+\tau,x(t_0+\tau)))}_{(I)}-\underbrace{(u(t_0,x(t_0))-P_{\le k}u(t_0,x(t_0)))}_{(II)}\\
&+\underbrace{(P_{\le k}u(t_0+\tau,x(t_0+\tau))-P_{\le k}u(t_0+\tau,x_{(k)}(t_0+\tau)))}_{(III)}-\underbrace{(P_{\le k}u(t_0,x(t_0))-P_{\le k}u(t_0,x_{(k)}(t_0)))}_{(IV)}\\
&+\underbrace{(P_{\le k}u(t_0+\tau,x_{(k)}(t_0+\tau))-P_{\le k}u(t_0,x_{(k)}(t_0)))}_{(V)}.
\end{split}
\end{align}

Term (I) and (II) can be estimated using the infinite sum:
\begin{align}
|u(t,x(t))-P_{\le k}u(t,x(t))|\le\sum_{h=k+1}^{\infty}|P_{h}u(t,x(t))|\lesssim(1+\delta)^{-\alpha k}\|u\|_{L^\infty_t\dot{C}^\alpha_x}\lesssim \tau^{\frac{\alpha}{1-\alpha}}\|u\|_{L^\infty_t\dot{C}^\alpha_x}^{\frac{1}{1-\alpha}}.
\end{align}
Note that the term (IV) is zero. For term (III), we need to estimate the difference between $x(t)$ and $x_{(k)}(t)$. From the trajectory equation, we can write
\begin{equation*}
\begin{split}
\partial_t(x(t)-x_{(k)}(t))
&=u(t,x(t))-P_{\le k}u(t,x_{(k)}(t))\\
&=(u(t,x(t))-P_{\le k}u(t,x(t)))+(P_{\le k}u(t,x(t))-P_{\le k}u(t,x_{(k)}(t))).
\end{split}
\end{equation*}
The first difference is estimated similarly to term (I) and (II). The second term is controlled using the fundamental theorem of calculus and the observation that $\|\nabla P_{\le k}u\|_{L^\infty}\lesssim (1+\delta)^{(1-\alpha)k}\|u\|_{L^\infty_t\dot{C}^\alpha_x}$. Summarizing these, we get
\begin{align}
|\partial_t(x(t)-x_{(k)}(t))|\le C(1+\delta)^{-\alpha k}\|u\|_{L^\infty_t\dot{C}^\alpha_x}+C(1+\delta)^{(1-\alpha)k}\|u\|_{L^\infty_t\dot{C}^\alpha_x}|x(t)-x_{(k)}(t)|.
\end{align}
Applying Gronwall's inequality and noting that $x(t_0)=x_{(k)}(t_0)$, we obtain the bound 
\begin{align}\label{Trajectory difference}
|x(t)-x_{(k)}(t)|\le C(1+\delta)^{-k}(\exp(C{(1+\delta)^{(1-\alpha)k}\|u\|_{L^\infty_t\dot{C}^\alpha_x}(t-t_0)})-1).
\end{align}
Taking $t=t_0+\tau$, we obtain the same bound $C\tau^{\frac{\alpha}{1-\alpha}}\|u\|_{L^\infty_t\dot{C}^\alpha_x}^{\frac{1}{1-\alpha}}$ for term (III).

For the last term (V), we estimate it using the material derivative:
\begin{align}
|P_{\le k}u(t_0+\tau,x_{(k)}(t_0+\tau))-P_{\le k}u(t_0,x_{(k)}(t_0))|\le\|(\partial_t+P_{\le k}u\cdot\nabla)P_{\le k}u\|_{L^\infty}\cdot\tau.
\end{align}
Thus, the problem reduces to estimating the material derivative $D_{\le k,t}P_{\le k}u$. To achieve this, we consider the evolution of Littlewood-Paley pieces. 
\begin{align*}
\partial_t P_{\le k}u+P_{\le k}u\cdot\nabla P_{\le k}u-\nu \Delta P_{\le k}u=-\nabla P_{\le k}p-\Div R_{\le k},
\end{align*}
where $R_{\le k}=P_{\le k}(u\otimes u)-P_{\le k}u\otimes P_{\le k}u$. Subtracting the $(k-1)$-th equation from the $k$-th equation, we derive
\begin{align}\label{Littlewood-Paley evolution}
\partial_t P_ku+P_{\le k}u\cdot\nabla P_ku-\nu \Delta P_ku=F_k=-P_ku\cdot\nabla P_{\le k-1}u-\nabla P_kp-\Div R_{\le k}+\Div R_{\le k-1}.
\end{align}
We now claim the forcing term estimate $\|F_k\|_{L^\infty}\lesssim (1+\delta)^{(1-2\alpha)k}\|u\|_{\dot{C}^\alpha}^2$. It can be derived in an easier way, but for later purposes, we rewrite the forcing term in alternative forms.

The pressure term has a good structure and can be estimated directly, but we do it differently by combining the high-high part also into commutator estimates for convenience. We can write
\begin{align}\label{pressure decomposition}
P_kp
&=P_k(-\Delta)^{-1}\partial_i\partial_jP_{\le k+2}(u^iu^j)\nonumber\\
&=P_k(-\Delta)^{-1}\partial_i\partial_jR_{\le k+2}^{ij}+P_k(-\Delta)^{-1}(\partial_jP_{\le k+2}u^i\cdot\partial_iP_{\le k+2}u^j).
\end{align}

The estimates for the terms $P_ku\cdot\nabla P_{\le k-1}u$ and $P_k(-\Delta)^{-1}(\partial_jP_{\le k+2}u^i\cdot\partial_iP_{\le k+2}u^j)$ are straightforward. It suffices to apply the commutator estimate \cite{CET94}: 
\begin{align}
\|R_{\le k}\|_{L^\infty}\lesssim (1+\delta)^{-2\alpha k}\|u\|_{\dot{C}^\alpha}^2.\label{Commutator base case}  
\end{align}
We give a proof for \eqref{Commutator base case}, which is not the shortest, but is useful for later use. Recall the useful decomposition exhibited in \cite{Ise23}: 
\begin{align}
R_{\le k}(x)&=P_{\le k}((u-P_{\le k}u(x))\otimes(u-P_{\le k}u(x)))(x)\\
&=P_{\le k}(P_{>k}u\otimes P_{>k}u)(x)\tag{HH}\\
&+P_{\le k}(P_{>k}u\otimes (P_{\le k}u-P_{\le k}u(x)))(x)+P_{\le k}((P_{\le k}u-P_{\le k}u(x))\otimes P_{>k}u)(x)\tag{HL,LH}\\
&+P_{\le k}((P_{\le k}u-P_{\le k}u(x))\otimes (P_{\le k}u-P_{\le k}u(x)))(x).\tag{LL}
\end{align}
We denote the terms by $R_{\le k}=R_{\le k,HH}+R_{\le k,HL}+R_{\le k,LH}+R_{\le k,LL}$. Furthermore, these terms can be written as:
\begin{equation}\label{Commutator decomposition}
\begin{split}
R_{\le k,HH}(x)&=\sum_{h>k}P_{\le k} (P_hu\otimes P_{\approx h}u),\\
R_{\le k,HL}(x)&=\int \varphi_{\le k}(y)(P_{\le k}u(x-y)-P_{\le k}(x))\otimes P_{\approx k}u(x-y)\,dy,\\
R_{\le k,LH}&=R_{\le k,HL}^T=\int \varphi_{\le k}(y)P_{\approx k}u(x-y)\otimes (P_{\le k}u(x-y)-P_{\le k}(x))\,dy,\\
R_{\le k,LL}(x)&=\int \varphi_{\le k}(y)(P_{\le k}u(x-y)-P_{\le k}(x))\otimes (P_{\le k}u(x-y)-P_{\le k}(x))\,dy,
\end{split}
\end{equation}
where $P_{\approx k}$ represents an operator of the form $=P_{[k+b_1,k+b_2]}$ with $|b_1|,|b_2|\le C_{\delta}$.
We can use the fundamental theorem of calculus to write 
\[P_{\le k}u(x-y)-P_{\le k}u(x)=-\int_0^1y\cdot\nabla P_{\le k}u(x-sy)\,ds.\]
Together with the bound $\|\varphi_{\le k}(y)y^{\otimes c}\|_{L^1}\lesssim_c (1+\delta)^{-ck}$, the estimate $\|R_{\le k}\|_{L^\infty}\lesssim (1+\delta)^{-2\alpha k}\|u\|_{\dot{C}^\alpha}^2$ readily follows.

It remains to bound the diffusion term. However, the Hölder bound only yields $\|\nu \Delta P_{\le k}u\|_{L^\infty}\lesssim \nu (1+\delta)^{(2-\alpha)k}\|u\|_{\dot{C}^\alpha}$, which provides the desired bound only for sufficiently large scales where $(1+\delta)^{-k}\ge \|u\|^{\frac{-1}{1+\alpha}}_{L^\infty\dot{C}^\alpha}\nu^{\frac{1}{1+\alpha}}$. At smaller scales $(1+\delta)^{-k}\le \|u\|_{L^\infty\dot{C}^\alpha}^{\frac{-1}{1+\alpha}}\nu^{\frac{1}{1+\alpha}}$, further analysis is required.

We employ again the maximum principle. Testing \eqref{Littlewood-Paley evolution} against $P_ku$ yields
\begin{align}
\frac{1}{2}\partial_t|P_ku|^2+\frac{1}{2}P_{\le k}u\cdot\nabla|P_ku|^2-\nu\Delta P_ku\cdot P_ku=F_k\cdot P_ku.
\end{align}
$\|P_ku\|_{L^\infty}(t)$ is a locally Lipschitz function in time, and for a.e. $t$ where it is differentiable, we have
\begin{align}
\frac{d}{dt}\|P_ku\|_{L^\infty}^2\le\sup_{x:|P_ku|^2(x)=\|P_ku\|_{L^\infty}}\partial_t|P_ku|^2(x).
\end{align}
For such maximum points $x$, $\nabla|P_ku|^2(x)=0$. Lemma \ref{Thin annulus lemma} and $d=2$ implies $-\Delta P_ku(x)\cdot P_ku(x)\ge c(1+\delta)^{2k}|P_ku|^2(x)$. We obtain
\[\frac{1}{2}\partial_t|P_ku|^2(x)\le -c\nu(1+\delta)^{2k}\|P_ku\|_{L^\infty}^2+\|F_k\|_{L^\infty}\|P_ku\|_{L^\infty}.\]
Thus, from the bound $\|F_k\|_{L^\infty}\lesssim (1+\delta)^{(1-2\alpha)k}\|u\|_{\dot{C}^\alpha}^2$, for a.e. $t$,
\begin{align}
\frac{d}{dt}\|P_ku\|_{L^\infty}\le -c\nu(1+\delta)^{2k}\|P_ku\|_{L^\infty}+C(1+\delta)^{(1-2\alpha)k}\|u\|_{\dot{C}^\alpha}^2.
\end{align}
By Gronwall's inequality,
\begin{align}
\|P_ku\|_{L^\infty}\le e^{-c\nu(1+\delta)^{2k}t}\|P_ku\|_{L^\infty}(0)+C(1+\delta)^{(1-2\alpha)k}\|u\|_{L^\infty\dot{C}^\alpha}^2\frac{1}{c\nu(1+\delta)^{2k}}(1-e^{-c\nu(1+\delta)^{2k}t}).
\end{align}
For times $t\ge a\|u\|_{L^\infty\dot{C}^\alpha}^{\frac{-2}{1+\alpha}}\nu^{\frac{1-\alpha}{1+\alpha}}$, the small-scale condition $(1+\delta)^{-k}\le \|u\|_{L^\infty\dot{C}^\alpha}^{\frac{-1}{1+\alpha}}\nu^{\frac{1}{1+\alpha}}$ ensures that $t\ge a\|u\|_{\dot{C}^\alpha}^{-1}(1+\delta)^{(\alpha-1)k}$. This leads to the bound
\[e^{-c\nu(1+\delta)^{2k}t}\|P_ku\|_{L^\infty}(0)\lesssim (c\nu (1+\delta)^{(1+\alpha)k}a\|u\|_{L^\infty\dot{C}^\alpha}^{-1})^{-1}\cdot (1+\delta)^{-\alpha k}\|u\|_{L^\infty\dot{C}^\alpha}\lesssim a^{-1}\nu^{-1}(1+\delta)^{(-1-2\alpha)k}\|u\|_{L^\infty\dot{C}^\alpha}^2.\]

Finally, we conclude the bound
\begin{align}\label{M.5 base case}
\|P_ku\|_{L^\infty}(t)\le C(a^{-1}+1)\|u\|_{L^\infty\dot{C}^\alpha}^2\nu^{-1}(1+\delta)^{(-1-2\alpha)k}.
\end{align}
Applying this estimate to the diffusion term, we obtain
\begin{align}
\|\nu\Delta P_ku\|_{L^\infty}(t)\le C(a^{-1}+1)\|u\|_{L^\infty\dot{C}^\alpha}^2(1+\delta)^{(1-2\alpha)k}.
\end{align}

From \eqref{Littlewood-Paley evolution}, we see that the material derivative obeys the same bound:
\begin{align}
\|(\partial_t+P_{\le k}u\cdot\nabla)P_ku\|_{L^\infty}(t)\le C(a^{-1}+1)\|u\|^2_{L^\infty\dot{C}^\alpha}(1+\delta)^{(1-2\alpha)k}.
\end{align}

To finish the proof, it remains to show that $(\partial_t+P_{\le k}u\cdot\nabla)P_{\le k}u$ also satisfies the same estimate. Consider the telescoping sum
\begin{align}
(\partial_t+P_{\le k}u\cdot\nabla)P_{\le k}u=\sum_{l\le k-1}\delta_{(l)}(\partial_t+P_{\le l}u\cdot\nabla)P_{\le l}u,
\end{align}
where $\delta_{(l)}$ is the ``increment'' defined by
\begin{align}\label{eq:increment}
\delta_{(l)}E_l=E_{l+1}-E_l,
\end{align}
for an expression $E_l$ indexed by $l$. We will use this notation a few times later.

Rewriting the increment as $(\partial_t+P_{\le l+1}u\cdot\nabla)P_{l+1}u+P_{l+1}u\cdot\nabla P_{\le l}u$, it follows that
\begin{align}
\|\delta_{(l)}(\partial_t+P_{\le l}u\cdot\nabla)P_{\le l}u\|_{L^\infty}\lesssim (a^{-1}+1)\|u\|^2_{L^\infty\dot{C}^\alpha}(1+\delta)^{(1-2\alpha)l}.
\end{align}
In the range $0<\alpha<\frac{1}{2}$, by the above estimates,
\[\|(\partial_t+P_{\le k}u\cdot\nabla)P_{\le k}u\|_{L^\infty}\lesssim(a^{-1}+1)\|u\|^2_{L^\infty\dot{C}^\alpha}\sum_{l\le k}(1+\delta)^{(1-2\alpha)l}\lesssim (a^{-1}+1)\|u\|^2_{L^\infty\dot{C}^\alpha}(1+\delta)^{(1-2\alpha)k}.\]
This gives the desired estimate for term (V). 

On the other hand, if $\alpha=\frac{1}{2}$, the only difference lies in the final summation, as we only have the following bound for the increments:
\[\|\delta_{(l)}(\partial_t+P_{\le l}u\cdot\nabla)P_{\le l}u\|\lesssim (a^{-1}+1)\|u\|^2_{L^\infty\dot{C}^\alpha}.\]
Since there are only finitely many $l\le k$, summing over $l\le k$ yields a power of $k$:
\begin{align}
\|(\partial_t+P_{\le k}u\cdot\nabla)P_{\le k}u\|\lesssim(a^{-1}+1)\|u\|_{L^\infty\dot{C}^\alpha}^2\max\{1,k\}.
\end{align}
From the relation $\tau\sim\|u\|_{L^\infty\dot{C}^\alpha}^{-1}(1+\delta)^{(\alpha-1)k}$, we have $\max\{1,k\}\sim C(1-\log^-(\|u\|_{L^\infty_t\dot{C}^\alpha_x}\tau))$, a logarithmic loss.

In summary, we have established Theorem \ref{Lagrangian estimate} in the range $0<\alpha\le \frac{1}{2}$ and $d=2$:
\begin{thm}\label{Lagrangian estimate less than one half}
Let $d=2$.

Let $u$ be an $L^\infty_tC^\alpha_x$ solution to \eqref{Navier-Stokes}, $0<\alpha\le \frac{1}{2}$, and $x(t)$ be a trajectory of $u$. Then, for any $a>0$ and $t_1,t_2\ge a\|u\|_{L^\infty_t\dot{C}^\alpha_x}^{\frac{-2}{1+\alpha}}\nu^{\frac{1-\alpha}{1+\alpha}}$, we have the following estimates:
\begin{enumerate}[label=(\roman*), font=\upshape]
    \item If $0<\alpha<\frac{1}{2}$,
    \begin{align}
    |x^\prime(t_1)-x^\prime(t_2)|\le C(a^{-1}+1)\|u\|^{\frac{1}{1-\alpha}}_{L^\infty_t\dot{C}^\alpha_x}|t_1-t_2|^{\frac{\alpha}{1-\alpha}}.
    \end{align}
    \item If $\alpha=\frac{1}{2}$,
    \begin{align}
    |x^\prime(t_1)-x^\prime(t_2)|\le C(a^{-1}+1)\|u\|^{\frac{1}{1-\alpha}}_{L^\infty_t\dot{C}^\alpha_x}|t_1-t_2|(1-\log^-(\|u\|_{L^\infty_t\dot{C}^\alpha_x}|t_1-t_2|)).
    \end{align}
\end{enumerate}
The constant $C$ is independent of $\nu>0$ and $a>0$.
\end{thm}

\subsection{The Trick in 3d}

Let $f_0=P_ku$ and $\lambda=(1+\delta)^k$.
As we investigate in Section \ref{sec:generalized thin annulus lemma}, it could happen that $-\Delta f_0\cdot f_0$ does not give us enough damping to use in the proof.
The goal in this subsection is to obtain the estimate \eqref{M.5 base case} in 3d, and the result follows from the same steps.

Following the discussion in Section \ref{sec:generalized thin annulus lemma}, if we consider the function $(-\Delta-\lambda^2)f_0$, then either this function or $f_0$ has a good instantaneous dissipation. However, bounding this function does not give much information to our problem. 

Instead, we consider the auxiliary function $f_1=\lambda^{-2}(-\Delta-(\lambda(1+2\delta))^2)f_0$. Then, we have the following three properties:
\begin{enumerate}
    \item The $L^\infty$ norms of $f_0,f_1$ are comparable up to a constant $C$ depending only on $\delta$:
    \begin{align}\label{eq:comparable}
        C^{-1}\|f_0\|_{L^\infty}\le \|f_1\|_{L^\infty}\le C\|f_0\|_{L^\infty}.
    \end{align}
    \item For each time $t$, either one of the following holds:
    \begin{align}\label{eq:either dissipates}
        \|(-\Delta-\lambda^2)f_0\|_{L^\infty}\le\epsilon\lambda^2\|f_0\|_{L^\infty}\, \text{ or }\ \|(-\Delta-\lambda^2)f_1\|_{L^\infty}\le\epsilon\lambda^2\|f_1\|_{L^\infty}.
    \end{align}
    \item $f_j$ satisfies a transport-diffusion equation with same forcing term estimates:
    \begin{align}\label{eq:same transport diffusion}
        \partial_tf_j+P_{\le k}u\cdot\nabla f_j-\nu\Delta f_j=F^{(j)},
    \end{align}
    where $\|F^{(j)}\|_{L^\infty}\lesssim \|u\|_{L^\infty\dot{C}^\alpha}^2(1+\delta)^{(1-2\alpha)k}$
\end{enumerate}

For \eqref{eq:comparable}, the key is that we consider the frequency $\lambda(1+2\delta)$, which is outside the frequency support of $f_0$. As a result, we can consider the inverse multiplier $\lambda^2(|\xi|^2-(\lambda(1+2\delta))^2)^{-1}\varphi(\frac{|\xi|-\lambda}{\lambda\delta})$, which has operator norm bounded by $C\delta^{-2}$ by a calculation similar to \eqref{eq:multiplier calculation}.

\eqref{eq:either dissipates} follows from Lemma \ref{lem:generalized thin annulus lemma} by setting $\phi(\delta)=2\delta$. We also present a direct proof here. Assume that $\|(-\Delta-\lambda^2)f_0\|_{L^\infty}\ge \epsilon\lambda^2\|f_0\|_{L^\infty}$, we want to conclude $\|(-\Delta-\lambda^2)f_1\|_{L^\infty}\le \epsilon\lambda^2\|f_1\|_{L^\infty}$ for sufficiently small $\delta$. Note that then
\[\|f_1\|_{L^\infty}\ge \lambda^{-2}\|(-\Delta-\lambda^2)f_0\|_{L^\infty}-(4\delta+4\delta^2)\|f_0\|_{L^\infty}\ge \frac{\epsilon}{2}\|f_0\|_{L^\infty},\]
if we pick $\delta$ sufficiently small. Next, we claim that
\begin{align}\label{eq:concrete proof}
\|(-\Delta-\lambda^2)(-\Delta-\lambda^2(1+2\delta)^2)f_0\|_{L^\infty}\le C\delta\lambda^4\|f_0\|_{L^\infty}.
\end{align}
Then, it follows that $\|(-\Delta-\lambda^2)f_1\|_{L^\infty}\le C\delta \epsilon^{-1}\lambda^2\|f_1\|_{L^\infty}$, which completes the assertion by taking $\delta$ sufficiently small again. For the bound, we analyze the multiplier
\begin{align}
m(\xi)=(|\xi|^2-\lambda^2)(|\xi|^2-\lambda^2(1+2\delta)^2)\varphi(\frac{|\xi|-\lambda}{\lambda\delta})
\end{align}
as in \eqref{eq:multiplier calculation}, we obtain $\|\partial_\xi^{\alpha}m\|_{L^2}\lesssim \lambda^{\frac{11}{2}}\delta^{\frac{5}{2}}(\lambda\delta)^{-|\alpha|}$ if $\delta\le 1$, and thus the $L^1$ norm of the kernel is bounded by $C\lambda^4\delta$.

For \eqref{eq:same transport diffusion}, we apply $\lambda^{-2}(-\Delta-\lambda^2(1+2\delta)^2)$ to the equation \eqref{Littlewood-Paley evolution}, which yields
\[\partial_tf_1+P_{\le k}u\cdot\nabla f_1-\nu\Delta f_1=F^{(1)},\]
where
\begin{align}
F^{(1)}=\lambda^{-2}(-\Delta-\lambda^2(1+2\delta)^2)F_k+\lambda^{-2}\Delta P_{\le k}u\cdot\nabla P_ku+2\lambda^{-2}\nabla P_{\le k}u\cdot\nabla^2 P_ku
\end{align}
Note that $\|F^{(1)}\|_{L^\infty}\lesssim \|u\|_{L^\infty\dot{C}^\alpha}^2(1+\delta)^{(1-2\alpha)k}$, which is the same estimate as $F^{(1)}:=F_k$. 

Finally, we finish the proof in 3d by showing that \eqref{eq:comparable}, \eqref{eq:either dissipates}, and \eqref{eq:same transport diffusion} give the estimate \eqref{M.5 base case} as in the 2d proof.

Thus, if we apply maximum principle to the equations for $f_0$ and $f_1$, we get
\[\partial_t\|f_j\|_{L^\infty}+\nu h_j(t)\|f_j\|_{L^\infty}\le C\|u\|_{L^\infty\dot{C}^\alpha}^2(1+\delta)^{(1-2\alpha)k},\]
where $h_j(t)$ are the instantaneous dissipation rates for $f_j$ at time $t$, defined by the following:
\begin{align}
h_j(t)=\sup_{x:|f_j|(t,x)=\|f_j\|_{L^\infty}(t)}\frac{-\Delta f_j\cdot f_j}{|f_j|^2}(t,x).
\end{align}
(We define it using supremum just for convenience. One can pick any $x$ achieving $\|f_j\|$.)
From the argument before, we know that $h_0+h_1\ge \frac{1}{2}(1+\delta)^{2k}$ for all time $t$. Here we chose $\epsilon=\frac{1}{2}$ and fixed a corresponding $\delta>0$. Now consider the quantity $Q=\|f_0\|_{L^\infty}\cdot\|f_1\|_{L^\infty}$. We can bound the evolution for it as follows:
\begin{equation}\label{eq:qunatity Q}
\begin{split}
\partial_t Q+\frac{1}{2}\nu(1+\delta)^{2k}Q
&\le C\|u\|_{L^\infty\dot{C}^\alpha}^2(1+\delta)^{(1-2\alpha)k}(\|f_0\|_{L^\infty}+\|f_1\|_{L^\infty})\\
&\le C\|u\|_{L^\infty\dot{C}^\alpha}^2(1+\delta)^{(1-2\alpha)k}Q^{1/2}\\
&\le \frac{1}{4}\nu(1+\delta)^{2k}Q+C\|u\|_{L^\infty\dot{C}^\alpha}^4\nu^{-1}(1+\delta)^{-4\alpha k}.
\end{split}
\end{equation}
We could then conclude that
\[Q(t)\le e^{-\frac{1}{4}\nu(1+\delta)^{2k}(t-t_0)}Q(t_0)+C\|u\|_{L^\infty\dot{C}^\alpha}^4\nu^{-2}(1+\delta)^{(-2-4\alpha)k}.\]
Since $Q(0)\le\|u\|_{\dot{C}^\alpha}^2(1+\delta)^{-2\alpha k}$ and $t\ge a\|u\|_{L^\infty\dot{C}^\alpha}^{\frac{-2}{1+\alpha}}\nu^{\frac{1-\alpha}{1+\alpha}}\ge a\|u\|_{\dot{C}^\alpha}^{-1}(1+\delta)^{(\alpha-1)k}$, the inequality $e^{-x}\lesssim x^{-2}$ yields
\begin{align}
Q(t)\le C(a^{-2}+1)\|u\|_{L^\infty\dot{C}^\alpha}^4\nu^{-2}(1+\delta)^{(-2-4\alpha)k}.
\end{align}
Since $\|f_0\|_{L^\infty}\lesssim Q^{1/2}$, this finishes the proof.

\subsection{Alternative Approach}\label{sec:alternative approach}

We now present our first argument using an $L^p$ energy estimate. We demonstrate this approach in 2d, while higher dimension can also be done using a suitable $L^p$ analogue of the generalized thin annulus lemma (Lemma \ref{lem:generalized thin annulus lemma}). 

Testing the equation \eqref{Littlewood-Paley evolution} against $|P_ku|^{p-2}P_ku$ and integrating over the spatial domain, we obtain

\[\partial_t\frac{1}{p}\int|P_ku|^p+\int (P_{\le k}u\cdot\nabla) P_ku\cdot P_ku|P_ku|^{p-2}-\nu\int\Delta P_ku\cdot P_ku|P_ku|^{p-2}=\int F_k\cdot P_ku|P_ku|^{p-2}.\]
Note that by the divergence-free condition,
\[\int (P_{\le k}u\cdot\nabla) P_ku\cdot P_ku|P_ku|^{p-2}=\frac{1}{p}\int P_{\le k}u\cdot\nabla|P_ku|^p=\frac{-1}{p}\int \Div P_{\le k}u\cdot |P_ku|^p=0.\]
By Lemma \ref{Instantaneous dissipation rate},
\[\int -\Delta P_ku\cdot P_ku|P_ku|^{p-2}\ge c(1+\delta)^{2k}\int|P_ku|^p.\]
Combining these and using Hölder's inequality, we obtain
\begin{align}
\partial_t\frac{1}{p}\int|P_ku|^p+c_p\nu(1+\delta)^{2k}\int|P_ku|^p\le \|F_k\|_{L^\infty}\cdot\int|P_ku|^{p-1}\le C_{\T^d}\|F_k\|_{L^\infty}(\int|P_ku|^p)^{\frac{p-1}{p}}.
\end{align}
Setting $g(t)=\|P_ku\|_{L^p}(t)$ and inserting the bound $\|F_k\|_{L^\infty}\lesssim (1+\delta)^{(1-2\alpha)k}\|u\|_{\dot{C}^\alpha}^2$ yields
\[\partial_tg+c\nu(1+\delta)^{2k}g\le C(1+\delta)^{(1-2\alpha)k}\|u\|^2_{\dot{C}^\alpha}.\]
Note that $g(0)=\|P_ku\|_{L^p}(0)\le C\|P_ku\|_{L^\infty}(0)\lesssim \|u\|_{L^\infty\dot{C}^\alpha}(1+\delta)^{-\alpha k}$. By Gronwall's inequality and the same argument before, we get
\begin{align}
\|P_ku\|_{L^p}(t)\lesssim (a^{-1}+1)\|u\|_{L^\infty\dot{C}^\alpha}^2\nu^{-1}(1+\delta)^{(-1-2\alpha)k},
\end{align}
with constants independent of large $p$. By taking $p\to\infty$, we recover the same estimate as before.

\begin{remark}
The second proof utilizes the divergence-free condition and the finiteness of $L^p$ norm, which uses more than the $L^\infty$ maximum principle. However, the $L^\infty$ method failed to give any improved bound if the annulus is not thin enough, since $-\Delta P_ku\cdot P_ku$ could vanishes at extreme points (see the remark after Lemma \ref{cp lemma}). On the other hand, Lemma \ref{cp lemma} tells us that we always have
\[\int-\Delta P_ku\cdot P_ku|P_ku|^{p-2}\ge \frac{c_\delta}{p}(1+\delta)^{2k}\int|P_ku|^p,\]
which together with Bernstein's inequality yields
\begin{align}
\|P_ku\|_{L^\infty}(t)\lesssim (1+\delta)^{\frac{dk}{p}}\|P_ku\|_{L^p}(t)\lesssim p(1+\delta)^{\frac{dk}{p}}\cdot (a^{-1}+1)\|u\|_{L^\infty\dot{C}^\alpha}^2\nu^{-1}(1+\delta)^{(-1-2\alpha)k}.
\end{align}
By choosing different $p=\max\{1,k\}$ depending on the scale, we can still get an estimate with a logarithmic loss using this method. Note that this method works in any dimension.
\end{remark}

\section{Lagrangian estimates in the full range of \texorpdfstring{$0<\alpha<1$}{0<alpha<1}: Set up}\label{sec:higher derivatives}

In the previous section, we proved Theorem \ref{Lagrangian estimate} in the range $0<\alpha\le \frac{1}{2}$, and we will set up the proof for the full range $0<\alpha<1$ in this section.

Again, many steps are the same as in \cite{Ise23,Ise25}, but we will provide a self-contained proof for the sake of completeness. 

\subsection{New Difficulties}

To obtain the full range $0<\alpha<1$, we need estimates on higher material derivatives $D_{\le k,t}^{m}P_ku$, where $D_{\le k,t}=\partial_t+P_{\le k}u\cdot\nabla$. Let us examine the new challenges that arise when estimating $D_{\le k,t}^2P_ku$ in the regime $\frac{1}{2}<\alpha\le\frac{2}{3}$.

The idea is to perform the same maximum principle trick on the function $D_{\le k,t}P_ku$. There are two problems: the function $D_{\le k,t}P_ku$ is not frequency localized in a thin annulus, and we need to commute $\nu\Delta$ and the material derivative $D_{\le k,t}$.

Applying the material derivative $D_{\le k,t}$ to the equation \eqref{Littlewood-Paley evolution} yields
\[D_{\le k,t}^2P_ku-D_{\le k,t}\nu\Delta P_ku=D_{\le k,t}F_k.\]
Assume for now that we know the forcing term estimate $\|D_{\le k,t}F_k\|_{L^\infty}\lesssim (a^{-1}+1)\|u\|^3_{L^\infty\dot{C}^\alpha}(1+\delta)^{(2-3\alpha)k}$. It suffices to bound the term $D_{\le k,t}\nu\Delta P_ku$. Consider the commutator 
\[[\nu\Delta,D_{\le k,t}]P_ku=\nu\Delta P_{\le k}u\cdot\nabla P_ku+2\nu\nabla P_{\le k}u\cdot\nabla^2 P_ku.\]
To obtain an estimate that aligns with the forcing term, we apply the following bound to one of the factors:
\[\|P_ku\|_{L^\infty}\lesssim (a^{-1}+1)\|u\|_{L^\infty\dot{C}^\alpha}^2\nu^{-1}(1+\delta)^{(-1-2\alpha)k}.\]
Note that the bound is proven for all $0<\alpha<1$ in the last subsection. We cannot apply it to $P_{\le k}u=\sum_{l\le k}P_l u$ that contains lower frequencies, since the sum would be dominated by the lowest frequencies and not $(1+\delta)^k$. However, both $\Delta P_{\le k}u$ and $\nabla P_{\le k}u$ have derivatives in front, and we simply use the Hölder bound:
\begin{align}
\begin{split}
&\|[\nu\Delta,D_{\le k,t}]P_ku\|_{L^\infty}\\
&\lesssim (a^{-1}+1)\|u\|^3_{L^\infty\dot{C}^\alpha}(\nu((1+\delta)^{(2-\alpha)k})(\nu^{-1}(1+\delta)^{-2\alpha k})+\nu((1+\delta)^{(1-\alpha)k})(\nu^{-1}(1+\delta)^{(1-2\alpha)k}))\\
&\lesssim(a^{-1}+1)\|u\|^3_{L^\infty\dot{C}^\alpha}(1+\delta)^{(2-3\alpha)k}.
\end{split}
\end{align}

We thus conclude that $D_{\le k,t}P_ku$ satisfies the forced transport-diffusion equation:
\[(\partial_t+P_{\le k}u\cdot\nabla)D_{\le k,t}P_ku-\nu\Delta D_{\le k,t}P_ku=\tilde{F}_k=[D_{\le k,t},\nu\Delta]P_ku+D_{\le k,t}F_k.\]
Next, we apply $P_{[k-2,k+2]}$ to the equation and obtain
\[(\partial_t+P_{\le k}u\cdot\nabla)P_{[k-2,k+2]}D_{\le k,t}P_ku-\nu\Delta P_{[k-2,k+2]} D_{\le k,t}P_ku=P_{[k-2,k+2]}\tilde{F}_k+[D_{\le k,t},P_{[k-2,k+2]}]D_{\le k,t}P_ku.\]
The new commutator term on the right hand side turns out to have a good estimate, and we can apply the maximum principle to $|P_{[k-2,k+2]}D_{\le k,t}P_ku|^2$.
Once we get the $L^\infty$ estimate for $P_{[k-2,k+2]}D_{\le k,t}P_ku$, observe that
\[D_{\le k,t}P_ku=D_{\le k,t}P_{[k-2,k+2]}P_ku= P_{[k-2,k+2]}D_{\le k,t}P_ku+[D_{\le k,t},P_{[k-2,k+2]}]P_ku.\]
The commutator term, again, has a good estimate, which implies the desired bound for $D_{\le k,t}P_ku$. The remaining arguments are essentially the same.

\subsection{Induction Scheme}

We now turn to the proof of the general case. Let us first outline the argument to determine what kind of estimates are required.

Let $m\ge 0$ be the integer such that $m<\frac{\alpha}{1-\alpha}\le m+1$. Note that qualitatively $u\in C^\infty$ for positive times, so the $(m+1)$-th derivative of a trajectory is given by the chain rule, and we only need to obtain quantitative estimates for its Hölder norm. As before, we choose $k$ such that $\tau\sim \|u\|_{L^\infty\dot{C}^\alpha}^{-1}(1+\delta)^{(\alpha-1)k}$, and consider the coarse trajectory $\partial_tx_{(k)}(t)=P_{\le k}u(t,x_{(k)}(t))$, $x_{(k)}(t_0)=x(t_0)$. We then obtain the decomposition analogous to \eqref{eq:difference decomposition}:
\begin{align}
\begin{split}
&D^m_tu(t_0+\tau,x(t_0+\tau))-D^m_tu(t_0,x(t_0))\\
&=\underbrace{(D^m_tu(t_0+\tau,x(t_0+\tau))-D_{\le k,t}^mP_{\le k}u(t_0+\tau,x(t_0+\tau)))}_{(I)}-\underbrace{(D^m_tu(t_0,x(t_0))-D_{\le k,t}^mP_{\le k}u(t_0,x(t_0)))}_{(II)}\\
&+\underbrace{(D_{\le k,t}^mP_{\le k}u(t_0+\tau,x(t_0+\tau))-D_{\le k,t}^mP_{\le k}u(t_0+\tau,x_{(k)}(t_0+\tau)))}_{(III)}\\
&-\underbrace{(D_{\le k,t}^mP_{\le k}u(t_0,x(t_0))-D_{\le k,t}^mP_{\le k}u(t_0,x_{(k)}(t_0)))}_{(IV)}\\
&+\underbrace{(D_{\le k,t}^mP_{\le k}u(t_0+\tau,x_{(k)}(t_0+\tau))-D_{\le k,t}^mP_{\le k}u(t_0,x_{(k)}(t_0)))}_{(V)},
\end{split}
\end{align}

where $D_t=(\partial_t+u\cdot\nabla)$.

Term (IV) is zero, while term (I), (II), and (III) are bounded by $\|D^m_tu-D^m_{\le k,t}P_{\le k}u\|_{L^\infty}$ and $\|\nabla D^m_{\le k,t}P_{\le k}u\|_{L^\infty}\cdot |x-x_{(k)}|$, respectively. The difference between the two trajectories obeys the bound \eqref{Trajectory difference} as before. By noting that $D_t^mu=\lim_{h\to\infty}D_{\le h,t}^mP_{\le h}u$ in the sense of distributions, we can rewrite both terms using increments (recall the notation \eqref{eq:increment}):
\begin{align*}
D^m_tu-D^m_{\le k,t}P_{\le k}u&=\sum_{h\ge k}(D^m_{\le h+1,t}P_{\le h+1}u-D^m_{\le h,t}P_{\le h}u)=\sum_{h\ge k}\delta_{(h)}D_{\le h,t}^mP_{\le h}u,\\
\nabla D^m_{\le k,t}P_{\le k}u&=\sum_{l\le k-1}\nabla(D^m_{\le l+1,t}P_{\le l+1}u-D^m_{\le l,t}P_{\le l}u)=\sum_{l\le k-1}\nabla\delta_{(l)}D_{\le l,t}^mP_{\le l}u.
\end{align*}

If we have the following estimate that will be proven in \eqref{M.2}:
\[\|\delta_{(k)} D_{\le k,t}^m P_{\le k}u\|_{L^\infty}\lesssim (a^{-m}+1)\|u\|_{L^\infty\dot{C}^\alpha}^{m+1}(1+\delta)^{(m(1-\alpha)-\alpha)k}.\]
Then, by noting $m(1-\alpha)-\alpha<0$ and $(m+1)(1-\alpha)>0$, the infinite sums both converge and yield
\begin{align}
\|D^m_tu-D^m_{\le k,t}P_{\le k}u\|&\lesssim(a^{-m}+1)\|u\|_{L^\infty\dot{C}^\alpha}^{m+1}(1+\delta)^{(m(1-\alpha)-\alpha)k},\\
\|\nabla D^m_{\le k,t}P_{\le k}u\|&\lesssim(a^{-m}+1)\|u\|_{L^\infty\dot{C}^\alpha}^{m+1}(1+\delta)^{(m+1)(1-\alpha)k}.
\end{align}
Term (V) is bounded by $\|D_{\le k,t}^{m+1}P_{\le k}u\|_{L^\infty}\cdot\tau$. Expressing it as a sum of increments over $l\le k-1$ and applying \eqref{M.2} again, we deduce that $\|D_{\le k,t}^{m+1} P_{\le k}u\|_{L^\infty}$ is bounded by:
\begin{equation}
\begin{dcases}
C(a^{-(m+1)}+1)\|u\|_{L^\infty\dot{C}^\alpha}^{m+2}(1+\delta)^{((m+1)(1-\alpha)-\alpha)k}  &,\quad\text{if }m<\frac{\alpha}{1-\alpha}<m+1,\\
C(a^{-(m+1)}+1)\|u\|_{L^\infty\dot{C}^\alpha}^{m+2}\max\{1,k\}&,\quad\text{if } \frac{\alpha}{1-\alpha}=m+1.
\end{dcases}
\end{equation}
Thus, the estimate \eqref{M.2} for the increments suffices to prove the Hölder bound. To obtain it, we propose an induction scheme of six estimates including \eqref{M.2}. We need to prove them jointly and sequentially, as the $m+1$ case of subsequent estimates might need the $m+1$ case of preceding ones, as well as the $\le m$ cases of all the estimates \eqref{M.1} to \eqref{M.6}.

Let us first introduce some notation. $P_{\lesssim k}u$ denotes any term of the form $P_{\le k+b}u$, with $b$ an integer that satisfies $|b|\le C_\alpha$. $P_{\approx k}u=P_{[k+b_1,k+b_2]}u$ with $|b_1|,|b_2|\le C_{\alpha}$. We allow $b_2<b_1$ so that it can also express a term $-P_{[k+b_2,k+b_1]}u$ with a negative sign. We need the estimate for $P_{\approx k}u$ mainly when dealing with the high-high frequency interactions in the commutator estimate. 

Denote $D_{\lesssim k,t}=\partial_t+P_{\lesssim k}u\cdot\nabla$. By a slight abuse of notation, we use $D_{\lesssim k,t}^m$ to express the product of $m$ such operators of the same form. We will not try to commute two operators of the same kind in the proof, so this ambiguity poses no issue. $\delta_{(k)}D_{\lesssim k}^{m}P_{\lesssim k}u$ represents the difference of two terms of the form $D_{\lesssim k}^{m}P_{\lesssim k}u$. 

Finally, $\perm(A^c,B^w)$ denotes any permutation of $c$ copies of operator $A$ and $w$ copies of operator $B$. For example, $\perm(A,B^2)$ could be $ABB$, $BAB$, or $BBA$.

We can now state the estimates:

\begin{thm}\label{Main lemma}
Let $u$ be a $L^\infty_tC^\alpha_x$ solution to \eqref{Navier-Stokes}. For any $a>0$ and any $t\ge a\|u\|^{\frac{-2}{1+\alpha}}_{L^\infty\dot{C}^\alpha}\nu^{\frac{1-\alpha}{1+\alpha}}$, the following estimates hold:

For all integers $0\le m<\frac{2\alpha}{1-\alpha}+1$,
\begin{align}
\|D_{\lesssim k,t}^{m} P_{\approx k}u\|_{L^\infty}&\lesssim(a^{-m}+1)\|u\|_{L^\infty\dot{C}^\alpha}^{m+1}(1+\delta)^{(m(1-\alpha)-\alpha)k},\tag{M1}\label{M.1}\\
\|\delta_{(k)} D_{\lesssim k,t}^m P_{\lesssim k}u\|_{L^\infty}&\lesssim (a^{-m}+1)\|u\|_{L^\infty\dot{C}^\alpha}^{m+1}(1+\delta)^{(m(1-\alpha)-\alpha)k},\tag{M2}\label{M.2}\\
\|\perm(\nabla,D_{\lesssim k,t}^m)P_{\lesssim k}u\|_{L^\infty}&\lesssim (a^{-m}+1)\|u\|_{L^\infty\dot{C}^\alpha}^{m+1}(1+\delta)^{(m+1)(1-\alpha)k},\tag{M3}\label{M.3}\\
\|[\nu\Delta,D_{\lesssim k,t}^m]P_{\approx k}u\|_{L^\infty}&\lesssim (a^{-m}+1)\|u\|_{L^\infty\dot{C}^\alpha}^{m+2}(1+\delta)^{((m+1)(1-\alpha)-\alpha)k}.\tag{M4}\label{M.4}
\end{align}
And, for all integers $0\le m<\frac{2\alpha}{1-\alpha}$,
\begin{align}
\|D_{\lesssim k,t}^m \nabla P_{\approx k}p\|_{L^\infty}+\|D_{\lesssim k,t}^m \nabla R_{\lesssim k}\|_{L^\infty}\lesssim (a^{-m}+1)&\|u\|_{L^\infty\dot{C}^\alpha}^{m+2}(1+\delta)^{((m+1)(1-\alpha)-\alpha)k},\tag{M5}\label{M.5}\\
\|D_{\le k,t}^{m}P_ku\|_{L^\infty}\lesssim (a^{-(m+1)}+1)&\|u\|_{L^\infty\dot{C}^\alpha}^{m+2}\nu^{-1}(1+\delta)^{-2k}(1+\delta)^{((m+1)(1-\alpha)-\alpha)k}.\tag{M6}\label{M.6}
\end{align}
\end{thm}

We will break the proof of this theorem into steps and finish the proof in the final section.

\section{Proof of Theorem \ref{Main lemma}}\label{sec:proof of main lemma}

We finish the technical calculations of the estimates in Theorem \ref{Main lemma}. It relies on various commutator formulas, and we will introduce them when we need to. The estimates \eqref{M.1} to \eqref{M.4} follows rather quickly from the induction hypothesis, while \eqref{M.5} and \eqref{M.6} require a more delicate computation.

\subsection{M1 to M4}

We will prove the inductive step of \eqref{M.1} to \eqref{M.4} in this subsection, which follows from some simple commutator formulas and the induction hypothesis. \eqref{M.5} and \eqref{M.6} are harder and we will postpone them to subsequent subsections.

Before starting the proof, we record some useful commutator formulas, whose proofs will be given in Appendix \ref{Commutator estimate proof}. In the following formulas, $\tr$ denotes a suitable trace operator for the given tensors.
\begin{lemma}
Let $X, Y:[0,T)\times\T^d\to\R^d$ be vector fields, and $f:[0,T)\times\T^d\to\R$ be a function. Define $D_t=\partial_t+X\cdot\nabla$. The following formulas hold:
\begin{align}
\perm(\nabla^{c+1},D_t^w)f=\sum C\tr \nabla^{c_0+1}D_t^{w_0}f\otimes\bigotimes_{i\ge 1}\nabla^{c_i+1}D_t^{w_i}X,\tag{C1}\label{C.1}
\end{align}
where $c,w,c_i,w_i\ge 0$, $\sum_{i\ge 0}c_i=c$, and $w_0+\sum_{i\ge 1}(w_i+1)=w$. Moreover, $[\nabla^{c+1},D_t^w]f$ can also be written as a sum of terms of the same expression, with $w_0<w$. On the other hand,
\begin{align}
\perm(Y\cdot\nabla,D_t^w)f=\sum C\tr\nabla D_t^{w_0}f\otimes D_t^{w_1}Y\otimes\bigotimes_{i\ge 2}\nabla D_t^{w_i}X,\tag{C2}\label{C.2}
\end{align}
where $w,w_i\ge 0$ and $w_0+w_1+\sum_{i\ge 2}(w_i+1)=w$.

\end{lemma}

Let us begin the proof. We proceed by induction on $m$. For $m=0$, the estimates \eqref{M.1}, \eqref{M.2}, \eqref{M.3}, and \eqref{M.4} are straightforward. \eqref{M.5} and \eqref{M.6} have been proven in \eqref{Commutator base case} and \eqref{M.5 base case}. (Strictly speaking, we have only proved \eqref{M.6} for dimension $2,3$ in previous sections, but the proof easily generalizes to higher dimensions and will be presented in the inductive proof of \eqref{M.6}.)

We now proceed to the inductive step. Suppose that for a given integer $0\le m<\frac{2\alpha}{1-\alpha}$, the six estimates in Theorem \ref{Main lemma} hold for all integers up to $m$. We want to prove that \eqref{M.1}, \eqref{M.2}, \eqref{M.3}, and \eqref{M.4} are true for $m+1$. If, additionally, $m+1<\frac{2\alpha}{1-\alpha}$, we also prove \eqref{M.5} and \eqref{M.6}.

\textit{Estimate} \eqref{M.1}: Recall the evolution of Littlewood-Paley pieces \eqref{Littlewood-Paley evolution}:
\[D_{\le k,t}P_ku-\nu \Delta P_ku=F_k=-P_ku\cdot\nabla P_{\le k-1}u-\nabla P_kp-\Div R_{\le k}+\Div R_{\le k-1}.\]
Applying $D_{\le k,t}^m$ to the equation, we can write
\[D_{\le k,t}^{m+1}P_ku=\nu\Delta D_{\le k,t}^{m}P_ku+[D_{\le k,t}^m,\nu\Delta]P_ku+D_{\le k,t}^m F_{k}.\]
The first term can be bounded by $C\nu (1+\delta)^{2k}\|D^m_{\le k,t}P_ku\|_{L^\infty}$, which we control using the order $m$ case of \eqref{M.6}. The second term, the pressure term, and the commutator term follow from the order $m$ cases of \eqref{M.4} and \eqref{M.5}. The remaining term, using the product rule, can be bounded by
\begin{align}
\|D_{\le k,t}^{m}(P_ku\cdot\nabla P_{\le k-1}u)\|_{L^\infty}\le\sum_{w_1+w_2=m} C\|D_{\le k,t}^{w_1}P_ku\|_{L^\infty}\|D_{\le k,t}^{w_2}\nabla P_{\le k-1}u\|_{L^\infty}.
\end{align}
The desired bound for $\|D_{\le k,t}^{m+1}P_ku\|_{L^\infty}$ then follows from \eqref{M.1} and \eqref{M.3} up to order $m$. 

It remains to show that this implies the general case $\|D_{\lesssim k,t}^{m+1}P_{\approx k}u\|_{L^\infty}$. By finite sums and shifts of indices, we may assume that we are differentiating $P_ku$ instead of $P_{\approx k}u$ . Observe that
\begin{align}
D_{\lesssim k,t}^{m+1}P_{k}u-D_{\le k,t}^{m+1}P_ku= \sum \perm(P_{\approx k}u\cdot\nabla,D_{\lesssim k,t}^{m})P_ku.
\end{align}
By applying \eqref{C.2}, we can express it as
\begin{align*}
\perm(P_{\approx k}u\cdot\nabla,D_{\lesssim k,t}^m)P_{k}u=\sum C\tr\nabla D_{\lesssim k,t}^{w_0}P_{k}u\otimes D_{\lesssim k,t}^{w_1}P_{\approx k}u\otimes\bigotimes_{i\ge 2}\nabla D_{\lesssim k,t}^{w_i}P_{\lesssim k}u,
\end{align*}
where $w_0+w_1+\sum_{i\ge 2}(w_i+1)= m$. This formula yields the bound
\begin{align}
\|\perm(P_{\approx k}u\cdot\nabla,D_{\lesssim k,t}^m)P_{k}u\|_{L^\infty}\lesssim\sum \|\nabla D_{\lesssim k,t}^{w_0}P_{ k}u\|_{L^\infty}\cdot\|D_{\lesssim k,t}^{w_1}P_{\approx k}u\|_{L^\infty}\cdot\prod_{i\ge 2}\|\nabla D_{\lesssim k,t}^{w_i}P_{\lesssim k}u\|_{L^\infty}.
\end{align}
In any case, all $w_i\le m$, so we can apply \eqref{M.1} and \eqref{M.3} up to order $m$ to obtain the bound
\begin{align}
\|D_{\lesssim k,t}^{m+1}P_{\approx k}u\|_{L^\infty}\lesssim \sum (a^{-\sum_{i\ge 0}w_i}+1)\|u\|^{m+2}_{L^\infty\dot{C}^\alpha}(1+\delta)^{((m+1)(1-\alpha)-\alpha)k}.
\end{align}
The worst case occurs when $w_0+w_1=m$, i.e., there are no $w_i$ terms for $i\ge 2$, which yields the desired bound. We omit similar calculations from now on.

\textit{Estimate} \eqref{M.2}:
Using a telescoping sum, we can write
\begin{align}
\delta_{(k)}D_{\lesssim k,t}^{m+1}P_{\lesssim k}u=D_{\lesssim k,t}^{m+1} P_{\approx k}u+\sum_{w} D_{\lesssim k,t}^w(P_{\approx k}u\cdot\nabla)D_{\lesssim k,t}^{m-w} P_{\lesssim k}u.
\end{align}
The first term can be bounded using the order $m+1$ case of \eqref{M.1}. The second term takes the form $\perm(P_{\approx k}u\cdot\nabla,D_{\lesssim k,t}^m)P_{\lesssim k}u$. Applying \eqref{C.2}, together with \eqref{M.1} and \eqref{M.3} up to order $m$, yields the desired bound.

\textit{Estimate} \eqref{M.3}: We apply \eqref{C.1} with $c=0$. All $c_i=0$, so we can write 
\[\perm(\nabla,D_{\lesssim k,t}^{m+1})P_{\lesssim k}u=\sum C\tr \nabla D_{\lesssim k,t}^{w_0}P_{\lesssim k}u\otimes\bigotimes_{i\ge 1}\nabla D_{\lesssim k,t}^{w_i}P_{\lesssim k}u,\]
where $w_0+\sum_{i\ge 1}(w_i+1)=m+1$. For terms where all $w_i\le m$, we apply \eqref{M.3} up to order $m$. In the case where $w_0=m+1$, which corresponds to the term $\nabla D_{\lesssim k,t}^{m+1}P_{\lesssim k}u$, we express it as a sum of increments:
\begin{align}
\nabla D_{\lesssim k,t}^{m+1}P_{\lesssim k}u=\sum_{l\le k}\nabla\delta_{(l)}D_{\lesssim l,t}^{m+1}P_{\lesssim l}u.
\end{align}
Using the newly established $m+1$ case of \eqref{M.2}, we obtain
\begin{align*}
\|\nabla D_{\lesssim k,t}^{m+1}P_{\lesssim k}u\|_{L^\infty}&\lesssim\sum_{l\le k}(a^{-(m+1)}+1)\|u\|_{L^\infty\dot{C}^\alpha}^{m+2}(1+\delta)^{(m+2)(1-\alpha)l}\\
&\lesssim (a^{-(m+1)}+1)\|u\|_{L^\infty\dot{C}^\alpha}^{m+2}(1+\delta)^{(m+2)(1-\alpha)k},
\end{align*}
since $(m+2)(1-\alpha)>0$.

\textit{Estimate} \eqref{M.4}:
We apply \eqref{C.1} and obtain the bound
\[\|[\nu\Delta,D_{\le k,t}^{m+1}]P_ku\|_{L^\infty}\lesssim \sum \nu\|\nabla^{c_0+1} D_{\le k,t}^{w_0}P_ku\|_{L^\infty}\cdot\prod_{i\ge 1}\|\nabla^{c_i+1} D_{\le k,t}^{w_i}P_{\le k}u\|_{L^\infty},\]
where $\sum_{i\ge 0}c_i=1$, $w_0+\sum_{i\ge 1}(w_i+1)=m+1$, but $w_0\le m$. Since all $w_i\le m$, applying \eqref{M.3} and \eqref{M.6} up to order $m$ eliminates the $\nu$ dependence and yields
\[\|[\nu\Delta,D_{\le k,t}^{m+1}]P_ku\|_{L^\infty}\lesssim(a^{-(\sum_{i\ge0}w_i+1)}+1)\|u\|^{m+3}_{L^\infty\dot{C}^\alpha}(1+\delta)^{((m+2)(1-\alpha)-\alpha)k}.\]
The worst case occurs when $w_0=m,w_1=0$, and there are no $w_i$ for $i\ge 2$, which yields the desired estimate.

\subsection{Forcing Term Estimates}

In this subsection, we will prove the order $m+1$ case of \eqref{M.5}, assuming the cases proven in the last subsection. \eqref{M.5} is essentially a forcing term estimate for the transport-diffusion equation satisfied by $D_{\lesssim k}^{m+1}P_{\approx k}u$. 

First, we need to observe a commutator formula for commuting material derivatives with convolution kernel, which relies on a delicate derivation and will play a central role in the estimates \eqref{M.5} and \eqref{M.6}. We present the detailed proof of this lemma below, as variants of this formula will be required when dealing with the full term $R_{\lesssim k}$.

\begin{lemma}\label{Convolution}
Let $X:[0,T)\times\T^d\to\R^d$ be a vector field, $K:\T^d\to\R$ be a convolution kernel, and $f:[0,T)\times\T^d\to\R$ be a function. Define $D_t=\partial_t+X\cdot\nabla$. The following formula holds:
\begin{align}
&D_t^w(x)\int_{\T^d} K(y)f(x-y)\,dy=\nonumber\\
&\sum C\tr\int_{\T^d} \nabla^{c}K(y)\otimes\bigotimes_{i=1}^{c}((D_t^{\tilde{w}_i}X)(x)-(D_t^{\tilde{w}_i}X)(x-y))\otimes\bigotimes_{i\ge 1}(\nabla D^{w_i}_t X)(x-y)\otimes D_t^{w_0}f(x-y)\,dy,\tag{C3}
\end{align}
where $w,w_i,\tilde{w}_i\ge0$, $0\le c\le w$, and $w_0+\sum_{i\ge 1}(w_i+1)+\sum_{i=1}^{c}(\tilde{w}_i+1)=w$. Moreover, $[D_t^w,K*]$ can also be written as a sum of the same expressions, with $w_0<w$.

As a result,
\begin{align}
\|D_t^w(K*f)\|_{L^\infty}\lesssim \sum\|D_t^{w_0}f\|_{L^\infty}\cdot\|\nabla^{c}K(y)y^{\otimes c}\|_{L^1}\cdot\prod_{i\ge 1}\|\nabla D_t^{w_i}X\|_{L^\infty},
\end{align}
where $w,w_i\ge0$, $0\le c\le w$, and $w_0+\sum_{i\ge 1}(w_i+1)=w$.
\begin{proof}
Clearly, the theorem is true for $w=0$. We proceed by induction on $w$ and apply $D_t(x)$ to the above expression. The product rule yields two terms:
\begin{align}
\tr\int \nabla^{c}K(y)\otimes\bigotimes_{i=1}^{c}((D_t^{\tilde{w}_i}X)(x)-(D_t^{\tilde{w}_i}X)(x-y))\otimes D_t(x)\left(\bigotimes_{i\ge 1}(\nabla D^{w_i}_t X)(x-y)\otimes D_t^{w_0}f(x-y)\right)\,dy,\label{product rule 1}\\
\tr\int \nabla^{c}K(y)\otimes D_t(x)\left(\bigotimes_{i=1}^{c}((D_t^{\tilde{w}_i}X)(x)-(D_t^{\tilde{w}_i}X)(x-y))\right)\otimes\bigotimes_{i\ge 1}(\nabla D^{w_i}_t X)(x-y)\otimes D_t^{w_0}f(x-y)\,dy.\label{product rule 2}
\end{align}

\textit{Term \eqref{product rule 1}:}
Write $D_t(x)=D_t(x-y)+(X(x)-X(x-y))\cdot\nabla$ and consider the two resulting terms separately.

If $D_t(x-y)$ falls on $D_t^{w_0}f(x-y)$ we increase $w_0$ by 1, and if it falls on $\nabla D_t^{w_i}X(x-y)$ we can write it as $D_t\nabla D_t^{w_i}X=\nabla D_t^{w_{i}+1}X-\nabla X\cdot\nabla D^{w_i}_tX$. The first term increases $w_i$ by 1, while the second term produces a new $w_{i^\prime}=0$. 

Consider the term $(X(x)-X(x-y))\cdot\nabla g(x-y)$, where $g=\bigotimes_{i\ge 1}(\nabla D^{w_i}_t X)\otimes D_t^{w_0}f$. Rewriting $\nabla_x(g(x-y))=-\nabla_y(g(x-y))$ and performing integration by parts with respect to $y$, the integral becomes
\[\pm\tr\int \nabla_y\left(\nabla^{c}K(y)\otimes\bigotimes_{i=1}^{c}((D_t^{\tilde{w}_i}X)(x)-(D_t^{\tilde{w}_i}X)(x-y))\otimes(X(x)-X(x-y))\right)\otimes g(x-y)\,dy.\]
If the derivative $\nabla_y$ falls on $\nabla^c K$, we increase $c$ by $1$. Note that we also introduce a new $\tilde{w}_{c+1}=0$.

If the derivative falls on one of $(D_t^{\tilde{w}_i}X)(x)-(D_t^{\tilde{w}_i}X)(x-y)$, the first term vanishes so that we produce a term $\nabla D_t^{\tilde{w}_i}X(x-y)$. This produces a new $w_{i^\prime}=\tilde{w}_i$, while the original $\tilde{w}_i$ becomes zero, as we have a new term $X(x)-X(x-y)$.

\textit{Term \eqref{product rule 2}:}
$D_t(x)$ falls on either one of the terms $(D_t^{\tilde{w}_i}X)(x)-(D_t^{\tilde{w}_i}X)(x-y)$, and we write it as
\[(D_t(x)(D_t^{\tilde{w}_i}X)(x)-D_t(x-y)(D_t^{\tilde{w}_i}X)(x-y))+(D_t(x-y)-D_t(x))(D_t^{\tilde{w}_i}X)(x-y).\]
The first term is simply $(D_t^{\tilde{w}_i+1}X)(x)-(D_t^{\tilde{w}_i+1}X)(x-y)$, which increases $\tilde{w}_i$ by 1.

The second term can be written as $(X(x-y)-X(x))\cdot\nabla D_t^{\tilde{w}_i}X(x-y)$, which again introduces a new $w_{i^\prime}=\tilde{w}_i$ and changes $\tilde{w}_i$ to zero. This finishes the proof of the formula. 

For the commutator $[D_t^w,K*]f$, we again proceed by induction on $w$. For $w=1$, it is clear from the above proof that the term where $D_t$ falls on $f$ is canceled out. For general $w$, use $[D_t^w,K*]=D_t[D_t^{w-1},K*]+[D_t,K*]D_t^{w-1}$. It is now clear by the same induction that every term has $w_0<w$.

Finally, the fundamental theorem of calculus yields 
\[D_t^{\tilde{w}_i}X(x)-D_t^{\tilde{w}_i}X(x-y)=\int_0^1y\cdot\nabla D_t^{\tilde{w}_i}X(x-sy)\,ds.\]
and the bound readily follows.
\end{proof}
\end{lemma}

From now on, we assume $m+1<\frac{2\alpha}{1-\alpha}$.

\textit{Estimate} \eqref{M.5}:
First, we use \eqref{pressure decomposition} to write:
\[\nabla P_{\approx k}p=P_{\approx k}(-\Delta)^{-1}\partial_i\partial_j\nabla R_{\lesssim k}^{ij}+P_{\approx k}\nabla(-\Delta)^{-1}(\partial_jP_{\lesssim k}u^i\partial_iP_{\lesssim k}u^j).\]
For the second term, we apply Lemma \ref{Convolution} to the operator $P_{\approx k}\nabla(-\Delta)^{-1}$ and bound
\begin{align}
\begin{split}
&\|D_{\lesssim k,t}^{m+1}P_{\approx k}\nabla(-\Delta)^{-1}(\partial_jP_{\lesssim k}u^i\cdot\partial_iP_{\lesssim k}u^j)\|_{L^\infty}\\
&\lesssim \sum\|D_{\lesssim k,t}^{w_0}(\partial_jP_{\lesssim k}u^i\cdot\partial_iP_{\lesssim k}u^j)\|_{L^\infty}\cdot\|\nabla^c K(y)y^{c}\|_{L^1}\cdot\prod_{i\ge 1}\|\nabla D_{\lesssim k,t}^{w_i}P_{\lesssim k}u\|_{L^\infty},
\end{split}
\end{align}
where $w_0+\sum_{i\ge 1}(w_i+1)=m+1$. Note that $\|\nabla^cK(y)y^{\otimes c}\|_{L^1}\lesssim_c(1+\delta)^{-k}$ here.
Each $w_i\le m+1$, and we can apply \eqref{M.3} up to order $m+1$ to obtain the bound
\[(a^{-\sum_{i\ge 0} w_i}+1)\|u\|^{m+3}_{L^\infty\dot{C}^\alpha}(1+\delta)^{((m+2)(1-\alpha)-\alpha)k}.\]
The sum $\sum_{i\ge 0}w_i$ is maximized when $w_0=m+1$, where we get the claimed estimate.

Applying Lemma \ref{Convolution} again to $P_{\approx k}(-\Delta)^{-1}\partial_i\partial_j$, we reduce the estimate for the pressure to the estimate for the commutator term:
\[\|D_{\lesssim k,t}^{w}\nabla R_{\lesssim k}\|_{L^\infty}\lesssim (a^{-w}+1)\|u\|^{w+2}_{L^\infty\dot{C}^\alpha}(1+\delta)^{((w+1)(1-\alpha)-\alpha)k}.\]
for all $w\le m+1$. Applying \eqref{C.1} yields the bound
\[\|D_{\lesssim k,t}^{w}\nabla R_{\lesssim k}\|_{L^\infty}\lesssim \sum \|\nabla D^{w_0}_{\lesssim k,t}R_{\lesssim k}\|_{L^\infty}\cdot\prod_{i\ge 1}\|\nabla D_{\lesssim k,t}^{w_i}P_{\lesssim k}u\|_{L^\infty}.\]
where $w_0+\sum_{i\ge 1}(w_i+1)=w$. Using \eqref{M.3} up to order $m$, the problem reduces to establishing the bound
\begin{align}
\|D_{\lesssim k,t}^{w} R_{\lesssim k}\|_{L^\infty}\lesssim (a^{-w}+1)\|u\|^{w+2}_{L^\infty\dot{C}^\alpha}(1+\delta)^{(w(1-\alpha)-2\alpha)k},
\end{align}
for all $w\le m+1$.

Up to an index shift, it suffices to bound $D_{\lesssim k,t}^{w}R_{\le k}$. Recall the decomposition \eqref{Commutator decomposition} of $R_{\le k}$:
\begin{align*}
R_{\le k,HH}(x)&=\sum_{h>k}P_{\le k}(P_{\approx h}u)^2(x),\\
R_{\le k,HL}(x)&=\int \varphi_{\le k}(y)P_{\approx k}u(x-y)\otimes (P_{\le k}u(x)-P_{\le k}u(x-y))\,dy,\\
R_{\le k,LL}(x)&=\int \varphi_{\le k}(y)(P_{\le k}u(x)-P_{\le k}u(x-y))\otimes (P_{\le k}u(x)-P_{\le k}u(x-y))\,dy.
\end{align*}

We first estimate the term $R_{\le k,HL}$. By modifying the proof of Lemma \ref{Convolution}, we obtain the expression 
\begin{align*}
D_{\lesssim k,t}^{w}(x)R_{\le k,HL}(x)\\
=\sum C\tr\int \nabla^{c}\varphi_{\le k}(y)&\otimes\bigotimes_{i=1}^{c}((D_{\lesssim k,t}^{\tilde{w}_i}P_{\lesssim k}u)(x)-(D_{\lesssim k,t}^{\tilde{w}_i}P_{\lesssim k}u)(x-y))\otimes\bigotimes_{i\ge 2}(\nabla D^{w_i}_{\lesssim k,t} P_{\lesssim k}u)(x-y)\\
&\otimes D_{\lesssim k,t}^{w_0}P_{\approx k}u(x-y)\otimes ((D_{\lesssim k,t}^{w_1}P_{\le k}u)(x)-(D_{\lesssim k,t}^{w_1}P_{\le k}u)(x-y))\,dy\\
+\sum C\tr\int \nabla^{c-1}\varphi_{\le k}(y)&\otimes\bigotimes_{i=1}^{c}((D_{\lesssim k,t}^{\tilde{w}_i}P_{\lesssim k}u)(x)-(D_{\lesssim k,t}^{\tilde{w}_i}P_{\lesssim k}u)(x-y))\otimes\bigotimes_{i\ge 2}(\nabla D^{w_i}_{\lesssim k,t} P_{\lesssim k}u)(x-y)\\
&\otimes D_{\lesssim k,t}^{w_0}P_{\approx k}u(x-y)\otimes (\nabla D_{\lesssim k,t}^{w_1}P_{\le k}u)(x-y)\,dy,
\end{align*}
where $w_0+w_1+\sum_{i\ge 2}(w_i+1)+\sum_{i\ge 1}(\tilde{w}_i+1)=w$, and $c\ge 1$ in the second term. This can be proved via a similar inductive argument. The formula yields the bound
\begin{align}
\|D^{w}_{\lesssim k,t}R_{\le k,HL}\|_{L^\infty}\lesssim \sum_{w_i,c}\|D_{\lesssim k,t}^{w_0}P_{\approx k}u\|\cdot\|\nabla D_{\lesssim k,t}^{w_1}P_{\le k}u\|\cdot\|\nabla^c \varphi_{\le k}(y)y^{\otimes c+1}\|_{L^1}\cdot\prod_{i\ge 2}\|\nabla D_{\lesssim k,t}^{w_i}P_{\lesssim k}u\|,
\end{align}
where $w_0+w_1+\sum_{i\ge 2}(w_i+1)=w$. We get the desired estimate for $w\le m+1$ by applying \eqref{M.1} and \eqref{M.3} up to order $m+1$. 

The term $R_{\le k,LL}$ can treated similarly. This time we have four terms in the formula, and the bound we get is
\[\|D^{w}_{\lesssim k,t}R_{\le k,LL}\|_{L^\infty}\lesssim \sum_{w_i,c}\|\nabla D_{\lesssim k,t}^{w_0}P_{\le k}u\|_{L^\infty}\cdot\|\nabla D_{\lesssim k,t}^{w_1}P_{\le k}u\|_{L^\infty}\cdot\|\nabla^c \varphi_{\le k}(y)y^{\otimes c+2}\|_{L^1}\cdot\prod_{i\ge 2}\|\nabla D_{\lesssim k,t}^{w_i}P_{\lesssim k}u\|_{L^\infty},\]
where $w_0+w_1+\sum_{i\ge 2}(w_i+1)=w$. This yields the same estimate.

The term $R_{\le k,HH}$ is more difficult. We must generate $D_{\le h,t}$ on the term $P_{\approx h}u$ to obtain the correct estimate. In \cite{Ise23}, he proposed a clever trick, which is to write
\[D_{\lesssim k,t}P_{\lesssim k}=P_{\lesssim k}(\partial_t+P_{\lesssim k}u\cdot\nabla)P_{\lesssim k}=P_{\lesssim k}(D_{\le h,t}-P_{[k+b,h]}u\cdot\nabla)P_{\lesssim k}=P_{\lesssim k}D_{\le h,t}P_{\lesssim k}-P_{\lesssim k} (P_{\approx k}u\cdot\nabla)P_{\lesssim k}.\]
Note that in the first equality, we apply a $P_{\lesssim k}$ for free since applying the material derivative only increase the frequency radius by some constant multiple. In the last equality, we can replace $P_{[k+b,h]}$ by $P_{\approx k}$ in the last term due to frequency interactions, as it is sandwiched by two $P_{\lesssim k}$.

Our goal is to show the following estimate:
\begin{align}
\|D_{\lesssim k,t}^{w}P_{\lesssim k}(P_{\approx h}u)^2\|_{L^\infty}\lesssim (a^{-w}+1)\|u\|^{w+2}_{L^\infty\dot{C}^\alpha}(1+\delta)^{(w(1-\alpha)-2\alpha)h},
\end{align}
for all $w\le m+1$. Since $m+1<\frac{2\alpha}{1-\alpha}$, after summing over $h>k$, we obtain the desired bound.

We employ this idea and write
\[D_{\lesssim k,t}^{w}P_{\lesssim k}(P_{\approx h}u)^2=D_{\lesssim k,t}^{w-1}P_{\lesssim k}D_{\le h,t}P_{\lesssim k}(P_{\approx h}u)^2+D^{w-1}_{\lesssim k,t}P_{\lesssim k}(P_{\approx k}u\cdot\nabla)P_{\lesssim k}(P_{\approx h}u)^2.\]

To simplify the second term, recall that applying Lemma \ref{Convolution} to the operator $P_{\lesssim k}$ yields:
\[\|D_{\lesssim k,t}^{w}P_{\lesssim k}g\|_{L^\infty}\lesssim \sum \|D_{\lesssim k,t}^{w_0}g\|_{L^\infty}\prod_{i\ge 1} \|D_{\lesssim k,t}^{w_i}\nabla P_{\lesssim k}u\|_{L^\infty},\]
where $w_0+\sum_{i\ge 1}(w_i+1)=w$. Observe that $(1+\delta)^{(w_i+1)(1-\alpha)k}\lesssim(1+\delta)^{(w_i+1)(1-\alpha)h}$. Consequently, we can estimate $D_{\lesssim k,t}^{w}P_{\lesssim k}g$ for $w\le m+1$ provided that we have the corresponding estimate for $D_{\lesssim k,t}^{w_0}g$ for $w_0\le m+1$.

From this, the estimate for the second term reduces to that of $D^{w}_{\lesssim k,t}(P_{\approx k}u\cdot\nabla)P_{\lesssim k}(P_{\approx h}u)^2$ for $w\le m$.
Observing that $D^w_{\lesssim k,t}(P_{\approx k}u\cdot\nabla)=\perm(P_{\approx k}u\cdot\nabla,D^w_{\lesssim k,t})$ and applying \eqref{C.2}, we obtain
\[\|D_{\lesssim k,t}^w(P_{\approx k}u\cdot\nabla)g\|_{L^\infty}\lesssim \sum_{w_i} \|\nabla D_{\lesssim k,t}^{w_0}g\|_{L^\infty}\cdot\|D_{\lesssim k,t}^{w_1}P_{\approx k}u\|_{L^\infty}\cdot\prod_{i\ge 2}\|\nabla D_{\lesssim k,t}^{w_i} P_{\lesssim k}u\|_{L^\infty},\]
where $w_0+w_1+\sum_{i\ge 2}(w_i+1)=w$. Since $g=P_{\lesssim k}(P_{\approx h}u)^2$ is frequency localized around $\lesssim (1+\delta)^k$, the $\nabla$ in the first term costs $(1+\delta)^k$. Multiplying this $(1+\delta)^k$ by the estimate for the second term gives the bound $(1+\delta)^{(w_1+1)(1-\alpha)k}\lesssim (1+\delta)^{(w_1+1)(1-\alpha)h}$. 

In summary, we reduce the problem to estimating the two terms $D_{\lesssim k,t}^wP_{\lesssim k}D_{\le h,t}P_{\lesssim k}(P_{\approx h}u)^2$ and $D_{\lesssim k,t}^wP_{\lesssim k}(P_{\approx h}u)^2$ for $w\le m$. Proceeding by downward induction on the power of $D_{\lesssim k,t}$ in front, it suffices to prove the following estimate:
\begin{align}
\|(P_{\lesssim k}D_{\le h,t})^wP_{\lesssim k}(P_{\approx h}u)^2\|_{L^\infty}\lesssim (a^{-w}+1)\|u\|^{w+2}_{L^\infty\dot{C}^\alpha}(1+\delta)^{(w(1-\alpha)-2\alpha)h},
\end{align}
for all $w\le m+1$. We now examine the operator $D_{\le h,t}P_{\lesssim k}$, and we need to be careful due to their different frequency supports. However, Lemma \ref{Convolution} yields an estimate of the form:
\[\|D_{\le h,t}^wP_{\lesssim k}g\|_{L^\infty}\le\sum\|D_{\le h,t}^{w_0}g\|_{L^\infty}\cdot\|\nabla^c\varphi_{\lesssim k}(y)y^{\otimes c}\|_{L^1}\cdot\prod_{i\ge 1}\|\nabla D_{\le h,t}^{w_i}P_{\le h}u\|_{L^\infty},\]
where $w_0+\sum_{i\ge 1}(w_i+1)=w$. For all $c\ge0$, we have the estimate $\|\nabla^c\varphi_{\lesssim k}(y)y^{\otimes c}\|_{L^1}\lesssim_c 1$. Note that if we had not performed integration by parts in Lemma \ref{Convolution}, then the derivatives on $K$ would act on the terms $g$ and $P_{\le h}u$, which have higher frequency and give a worse estimate. 

Finally, it remains to show that
\[\|D_{\le h,t}^w(P_{\approx h}u)^2\|_{L^\infty}\lesssim (a^{-w}+1)\|u\|^{w+2}_{L^\infty\dot{C}^\alpha}(1+\delta)^{(w(1-\alpha)-2\alpha)h},\]
for all $w\le m+1$. The bound follows directly from the product rule and \eqref{M.1}.

\subsection{Maximum Principle}

In this final subsection, we will prove \eqref{M.6} using the same maximum principle trick applied to $P_{[k-2,k+2]}D_{\le k}^{m+1}P_{k}u$.

\textit{Estimate} \eqref{M.6}:
First, note that if $(1+\delta)^{-k}\ge \|u\|_{L^\infty\dot{C}^\alpha}^{\frac{-1}{1+\alpha}}\nu^{\frac{1}{1+\alpha}}$, then $\nu^{-1}(1+\delta)^{-2k}\ge \|u\|_{L^\infty\dot{C}^\alpha}^{-1}(1+\delta)^{(\alpha-1)k}$, and the bound \eqref{M.1} is already stronger. Consequently, we may assume $(1+\delta)^{-k}\le \|u\|_{L^\infty\dot{C}^\alpha}^{\frac{-1}{1+\alpha}}\nu^{\frac{1}{1+\alpha}}$.

Recall again the evolution of Littlewood-Paley pieces \eqref{Littlewood-Paley evolution}:
\[D_{\le k,t} P_ku-\nu \Delta P_ku=F_k=-P_ku\cdot\nabla P_{\le k-1}u-\nabla P_kp-\Div R_{\le k}+\Div R_{\le k-1}.\]
Applying $P_{[k-2,k+2]}D_{\le k,t}^{m+1}$ to it, we obtain
\begin{align}
\begin{split}
&D_{\le k,t}P_{[k-2,k+2]}D_{\le k,t}^{m+1}P_ku-\nu\Delta P_{[k-2,k+2]}D_{\le k,t}^{m+1} P_ku\\
&=\tilde{F}_k=P_{[k-2,k+2]}D_{\le k,t}^{m+1}F_k+P_{[k-2,k+2]}[D_{\le k,t}^{m+1},\nu\Delta]P_ku+[D_{\le k,t},P_{[k-2,k+2]}]D_{\le k,t}^{m+1}P_ku.
\end{split}
\end{align}

From Lemma \ref{Convolution}, we know that
\[\|[D_{\le k,t},P_{[k-2,k+2]}]f\|_{L^\infty}\lesssim \|\nabla P_{\le k}u\|_{L^\infty}(\|\varphi_{[k-2,k+2]}(y)\|_{L^1}+\|\nabla\varphi_{[k-2,k+2]}(y)y\|_{L^1})\|f\|_{L^\infty}.\]
Using the product rule, \eqref{M.1}, \eqref{M.3}, \eqref{M.4}, and \eqref{M.5} up to order $m+1$, we obtain the bound
\begin{align}\label{Final forcing term}
\|\tilde{F}_k\|_{L^\infty}\lesssim(a^{-(m+1)}+1)\|u\|_{L^\infty\dot{C}^\alpha}^{m+3}(1+\delta)^{((m+2)(1-\alpha)-\alpha)k}.
\end{align}
Now let 
\[f_j=(1+\delta)^{-2jk}(-\Delta-(1+\delta)^{2k+8})^jP_{[k-2,k+2]}D_{\le k,t}^{m+1}P_ku\]
for $j=0,\dots,j_0(d)$. We have the following properties analogous to \eqref{eq:comparable}, \eqref{eq:either dissipates}, and \eqref{eq:same transport diffusion}. The proofs are the same.

\begin{enumerate}
    \item The $L^\infty$ norms of $f_j$ are comparable up to a constant $C$ depending only on $\delta$: For all $j,j^\prime$,
    \begin{align}\label{eq:comparable 2}
        \|f_j\|_{L^\infty}\le C\|f_{j^\prime}\|_{L^\infty}.
    \end{align}
    \item For each time $t$, either one of $f_j$ satisfies the following:
    \begin{align}\label{eq:either dissipates 2}
        \|(-\Delta-(1+\delta)^{2k})f_j\|_{L^\infty}\le\epsilon(1+\delta)^{2k}\|f_j\|_{L^\infty}.
    \end{align}
    \item $f_j$ satisfies a transport-diffusion equation with same forcing term estimates:
    \begin{align}\label{eq:same transport diffusion 2}
        \partial_tf_j+P_{\le k}u\cdot\nabla f_j-\nu\Delta f_j=F^{(j)},
    \end{align}
    where $\|F^{(j)}\|_{L^\infty}\lesssim (a^{-(m+1)}+1)\|u\|_{L^\infty\dot{C}^\alpha}^{m+3}(1+\delta)^{((m+2)(1-\alpha)-\alpha)k}$.
\end{enumerate}

When applying Lemma \ref{lem:generalized thin annulus lemma}, we choose $\epsilon=\frac{1}{2}$ and fix a $\delta=\delta(d)$ such that \eqref{eq:either dissipates 2} holds for functions whose frequency is supported in $\{\xi:(1+\delta)^{k-3}\le|\xi|\le (1+\delta)^{k+3}\}$.

Apply maximum principle to the equation for each $f_j$, we obtain
\[\partial_t\|f_j\|_{L^\infty}+\nu (1+\delta)^{2k} h_j(t)\|f_j\|_{L^\infty}\le C(a^{-(m+1)}+1)\|u\|_{L^\infty\dot{C}^\alpha}^{m+3}(1+\delta)^{((m+2)(1-\alpha)-\alpha)k}\]
with some $h_j\ge0$ and $\sum h_j\ge \frac{1}{2}$ for all times.
Consider the quantity $Q=\prod_{j=0}^{j_0(d)}\|f_j\|_{L^\infty}$ and argue as in \eqref{eq:qunatity Q}, we obtain
\[\partial_tQ+c\nu (1+\delta)^{2k}Q\le C(\nu(1+\delta)^{2k})^{-j_0(d)}\left((a^{-(m+1)}+1)\|u\|_{L^\infty\dot{C}^\alpha}^{m+3}(1+\delta)^{((m+2)(1-\alpha)-\alpha)k}\right)^{j_0(d)+1}.\]
Solving the ODE inequality, we see
\[Q(t)\le e^{-c\nu(1+\delta)^{2k}(t-t_1)}Q(t_1)+\left((a^{-(m+1)}+1)\|u\|_{L^\infty\dot{C}^\alpha}^{m+3}\nu^{-1}(1+\delta)^{-2k}(1+\delta)^{((m+2)(1-\alpha)-\alpha)k}\right)^{j_0(d)+1}.\]
Set $t_1=\frac{1}{2}a\|u\|_{L^\infty\dot{C}^\alpha}^{\frac{-2}{1+\alpha}}\nu^{\frac{1-\alpha}{1+\alpha}}$ and $t\ge a\|u\|_{L^\infty\dot{C}^\alpha}^{\frac{-2}{1+\alpha}}\nu^{\frac{1-\alpha}{1+\alpha}}$. Since the inductive estimates are valid for time $t_1$ (we have the freedom of choosing $a$), we can bound $Q(t_1)$ using $(j_0(d)+1)$-th power of \eqref{M.1}. Combining with $e^{-x}\lesssim x^{-(j_0(d)+1)}$, we can finally bound
\begin{align}
\|P_{[k-2,k+2]}D_{\le k,t}^{m+1}P_ku\|_{L^\infty}(t)\lesssim Q^{1/(j_0(d)+1)}\lesssim (a^{-(m+2)}+1)\|u\|_{L^\infty\dot{C}^\alpha}^{m+3}\nu^{-1}(1+\delta)^{-2k}(1+\delta)^{((m+2)(1-\alpha)-\alpha)k}.
\end{align}

Finally, note that
\[D_{\le k,t}^{m+1}P_ku=D_{\le k,t}^{m+1}P_{[k-2,k+2]}P_ku=[D_{\le k,t}^{m+1},P_{[k-2,k+2]}]P_ku+P_{[k-2,k+2]}D_{\le k,t}^{m+1}P_ku.\]
Using Lemma \ref{Convolution}, we can bound
\[\|[D_{\le k,t}^{m+1},P_{[k-2,k+2]}]P_ku\|_{L^\infty}\lesssim\sum\|D_{\le k,t}^{w_0}P_ku\|_{L^\infty}\cdot\|\nabla^c\varphi_{[k-2,k+2]}(y)y^{\otimes c}\|\cdot\prod_{i\ge 1}\|\nabla D_{\le k,t}^{w_i}P_{\le k}u\|,\]
where $w_0+\sum_{i\ge 1}(w_i+1)=m+1$, but $w_0\le m$. We use \eqref{M.6} up to order $m$ for the first term and $\eqref{M.3}$ up to order $m$ to bound
\[\|[D_{\le k,t}^{m+1},P_{[k-2,k+2]}]P_ku\|_{L^\infty}\lesssim(a^{-(m+1)}+1)\|u\|_{L^\infty\dot{C}^\alpha}^{m+3}\nu^{-1}(1+\delta)^{-2k}(1+\delta)^{((m+2)(1-\alpha)-\alpha)k},\]
which is no worse than the desired estimate. This completes the induction. \qed

We conclude the paper with a remark.

\begin{remark}
In \cite{Ise23}, many estimates are handled using iterated commutators. While we rely on a more direct approach—namely, rearranging the terms into suitable forms without using these concepts—iterated commutators provide a finer and more robust tool, which we also record here for completeness. Denoting $[A,]B=[A,B]=AB-BA$ and $[A,]^wB=[A,]([A,]^{w-1}B)$, we have the following estimates:
\begin{align}
\|[D_t,]^wT_K\|_{L^p\to L^p}\lesssim \sum\|\nabla^{c}K(y)y^{\otimes c}\|_{L^1}\cdot\prod_{i\ge 1}\|\nabla D_t^{w_i}X\|_{L^\infty},
\end{align}
where $w,w_i\ge0$, $0\le c\le w$, and $\sum_{i\ge 1}(w_i+1)=w$. Here $T_Kf(x)=\int_{\T^d} K(y)f(x-y)\,dy$ denotes the convolution operator.

Furthermore, we can express $[D_t,]^w(Y\cdot\nabla)=Z\cdot\nabla$, with
\begin{align}
\|Z\|_{L^\infty}\lesssim \sum\|D_t^{w_0}Y\|_{L^\infty}\prod_{i\ge 1}\|\nabla D_t^{w_i}X\|_{L^\infty},
\end{align}
where $w,w_i\ge0$ and $w_0+\sum_{i\ge 1}(w_i+1)=w$.

To prove the first bound, we modify the proof of Lemma \ref{Convolution} to express $([D_t,]^{w}T_K)f$ as a sum of terms of the following form
\[\tr\int \nabla^{c}K(y)\otimes\bigotimes_{i=1}^{c}((D_t^{\tilde{w}_i}X)(x)-(D_t^{\tilde{w}_i}X)(x-y))\otimes\bigotimes_{i\ge 1}(\nabla D^{w_i}_t X)(x-y)\otimes f(x-y)\,dy,\]
where $w,w_i,\tilde{w}_i\ge0$, $0\le c\le w$, and $\sum_{i\ge 1}(w_i+1)+\sum_{i=1}^{c}(\tilde{w}_i+1)=w$. Applying $[D_t,]$ to the above expression, we follow the exact same steps in Lemma \ref{Convolution}. The only difference is that the term arising when $D_t(x-y)$ acts on $f(x-y)$ cancels out. Consequently, no derivatives fall on $f$, and the bound readily follows. 
For the second bound, we proceed an induction on $w$ for the expression:
\[Z=\sum C\tr D_t^{w_0}Y\otimes\bigotimes_{i\ge 1}\nabla D_t^{w_i}X,\]
where $w,w_i\ge0$ and $w_0+\sum_{i\ge 1}(w_i+1)=w$.
Noting that $[D_t,](Z\cdot\nabla)=(D_tZ-Z\cdot\nabla X)\cdot\nabla$, the proof then follows from the identity $D_t\nabla=\nabla D_t-\nabla X\cdot\nabla$.
    
\end{remark}

\appendix

\section{Appendix}

\begin{lemma}\label{Parabolic estimate}
Let $u:[0,T)\times\T^d\to\R$ be a solution to the following equation:
\begin{equation}
\begin{dcases}
\partial_tu-\nu\Delta u=f,\\
u(0,x)=0.
\end{dcases}
\end{equation}
Then, we have the following estimates:
\begin{align}
\|\partial_tu\|_{L^\infty_t\dot{C}^\alpha_x}\le C\|f\|_{L^\infty_t\dot{C}^\alpha_x},\\
\|\partial_tP_ku\|_{L^\infty_{t,x}}\le C\|P_kf\|_{L^\infty_{t,x}}.
\end{align}
The constant $C=C_{d,\alpha}$ is independent of $\nu>0$.
\begin{proof}
We first prove the Hölder estimate. We view $u$ and $f$ as $(2\pi\Z)^d$-periodic functions on the whole $\R^d$, with the $C_t\dot{C}^\alpha_x$-norm comparable.

It is equivalent to bound the term $\nu\Delta u$. By Duhamel's formula,
\[\nu\Delta u(t,x)=\lim_{\epsilon\to0}\int_\epsilon^t\int K_{\nu}(\tau,x-y)f(t-\tau,y)\,dy\,d\tau,\]
where $K_\nu(\tau,y)$ is the kernel of the operator $\nu\Delta e^{\nu \tau\Delta}$. Note that $\|K_\nu\|_{L^1_y}(\tau)=\frac{c}{\tau}$.

To estimate $\nu\Delta u(t,x+h)-\nu\Delta u(t,x)$, we split the time interval into $(0,\delta)$ and $(\delta,t)$. For the first interval, we use $\int K_v(s,x-y)\,dy=0$ to insert a difference:
\begin{equation*}
\begin{split}
&\int_{0}^{\delta}\int K_\nu(\tau,x+h-y)f(t-\tau,y)\,dy\,d\tau-\int_{0}^{\delta}\int K_\nu(\tau,x-y)f(t-\tau,y)\,dy\,d\tau\\
&=\int_{0}^{\delta}\int K_\nu(\tau,x+h-y)(f(t-\tau,y)-f(t-\tau,x+h))\,dy\,d\tau-\int_{0}^{\delta}\int K_\nu(\tau,x-y)(f(t-\tau,y)-f(t-\tau,x))\,dy\,d\tau\\
&\le 2\|f\|_{L^\infty_t\dot{C}^\alpha_x}\int_0^\delta\int|K_\nu(\tau,y)|\cdot|y|^{\alpha}\,dy\,d\tau\\
&\lesssim\|f\|_{L^\infty_t\dot{C}^{\alpha}_x}\int_0^{\delta}(\nu \tau)^{\alpha/2}\frac{d\tau}{\tau}\lesssim\|f\|_{L^\infty_t\dot{C}^{\alpha}_x}(\nu\delta)^{\alpha/2},
\end{split}
\end{equation*}
where we use the homogeneity of the kernel to calculate $\|K_{\nu}|y|^\alpha\|_{L^1_y}(\tau)=c(\nu\tau)^{\alpha/2}\tau^{-1}$.

For $(\delta,t)$, we also need some smoothness of the kernel:
\begin{equation*}
\begin{split}
&\int_{\delta}^t\int(K_{\nu}(\tau,x+h-y)-K_\nu(\tau,x-y))f(t-\tau,y)\,dy\,d\tau\\
&=\int_0^1\int_{\delta}^t\int(\nabla K_{\nu}(\tau,x+\theta h-y)\cdot h) f(t-\tau,y)\,dy\,d\tau\,d\theta\\
&=\int_0^1\int_{\delta}^t\int(\nabla K_{\nu}(\tau,x+\theta h-y)\cdot h) (f(t-\tau,y)-f(t-\tau,x+\theta h))\,dy\,d\tau\,d\theta\\
&\le \|f\|_{L^\infty_t\dot{C}^\alpha_x}|h|\int_{\delta}^t\int|\nabla K_\nu(\tau,y)|\cdot|y|^{\alpha}\,dy\,d\tau\\
&\lesssim\|f\|_{L^\infty_t\dot{C}^\alpha_x}|h|\int_{\delta}^t(\nu \tau)^{(\alpha-1)/2}\frac{d\tau}{\tau}\\
&\lesssim\|f\|_{L^\infty_t\dot{C}^\alpha_x}|h|(\nu\delta)^{(\alpha-1)/2},
\end{split}
\end{equation*}
where we use $\|\nabla K_{\nu}|y|^\alpha\|_{L^1}(\tau)=c(\nu\tau)^{(\alpha-1)/2}\tau^{-1}$.
Choosing $\delta$ such that $(\nu\delta)^{1/2}=|h|$ balances the two terms, and we obtain the bound $C\|f\|_{L^\infty_t\dot{C}^\alpha_x}|h|^{\alpha}$.

The $L^{\infty}_{t,x}$ estimate for $\partial_t P_ku$ follows readily from choosing any $0<\beta<1$ and applying the Hölder estimate:
\[\|\partial_tP_ku\|_{L^\infty_{t,x}}\lesssim_{\beta} 2^{-\beta k}\|\partial_t P_ku\|_{L^\infty_t\dot{C}^\beta_x}\lesssim 2^{-\beta k}\|P_kf\|_{L^\infty_t\dot{C}^\beta_x}\lesssim_{\beta}\|P_kf\|_{L^\infty_{t,x}}.\]
\end{proof}
\end{lemma}

\begin{lemma}\label{Parabolic estimate 2}
Denote by $e^{t\Delta}f$ the solution to the free heat equation with initial data $f$, then we have the following estimates:
\begin{align}
\|e^{t\Delta}f-f\|_{L^\infty}\lesssim t^{\frac{\alpha}{2}}\|f\|_{\dot{C}^\alpha_{*}},\label{lem:A.2.1}\\
\|e^{t\Delta}f\|_{\dot{C}^\beta_{*}}\lesssim t^{\frac{-(\beta-\gamma)}{2}}\|f\|_{\dot{C}^\gamma_{*}},\label{lem:A.2.2}
\end{align}
where $0<\alpha<2$, $\beta>\gamma>0$.
\begin{proof}
For \eqref{lem:A.2.1}, we bound
\[\|e^{t\Delta}f-f\|_{L^\infty}\le \sum_{t2^{2k}\le 1}\|(e^{t\Delta}-1)P_kf\|_{L^\infty}+\sum_{t2^{2k}\ge 1}(\|e^{t\Delta}P_kf\|_{L^\infty}+\|P_kf\|_{L^\infty})\]
When $t2^{2k}\le 1$, investigate the multiplier
\[m_k(\xi)=(e^{-t|\xi|^2}-1)\varphi(\frac{|\xi|}{2^k})\]
From $|e^{-t|\xi|^2}-1|\le Ct|\xi|^2$, and the condition $2^kt\le 2^{-k}$, we know that $\|\partial_\xi^l m_k(\xi)\|_{L^2}\lesssim t2^{2k}\cdot 2^{-|l|k}\cdot 2^{\frac{d}{2}k}$ for every multiindex $l$. Similar to the argument \eqref{eq:multiplier calculation}, we obtain
\[\|(e^{t\Delta}-1)P_kf\|_{L^\infty}\lesssim t2^{2k}\|P_kf\|_{L^\infty}.\]
From this and $\|e^{t\Delta}P_kf\|_{L^\infty}\le \|P_kf\|_{L^\infty}$, we finally obtain
\[\|e^{t\Delta}f-f\|_{L^\infty}\lesssim\sum_{2^k\le t^{-1/2}}t2^{(2-\alpha)k}\|f\|_{\dot{C}^\alpha_*}+\sum_{2^k\ge t^{-1/2}}2^{-\alpha k}\|f\|_{\dot{C}^\alpha_*}\lesssim t^{\frac{\alpha}{2}}\|f\|_{\dot{C}^\alpha_*}\]

For \eqref{lem:A.2.2}, we use the Littlewood-Paley characterization of Hölder-Zygmund norms:
\[\|e^{t\Delta}f\|_{\dot{C}^\beta_*}\sim\sup_k 2^{\beta k}\|e^{t\Delta} P_kf\|_{L^\infty}=\sup_k2^{\gamma k}\|P_{[k-2,k+2]}(\frac{|\nabla|}{2^k})^{\gamma-\beta}P_{[k-2,k+2]}|\nabla|^{\beta-\gamma}e^{t\Delta} P_kf\|_{L^\infty}.\]
Since $\|P_{[k-2,k+2]}(\frac{|\nabla|}{2^k})^{\gamma-\beta}\|\lesssim 1$ and $\|P_{[k-2,k+2]}|\nabla|^{\beta-\gamma}e^{t\Delta}\|\lesssim t^{\frac{-(\beta-\gamma)}{2}}$, we obtain
\[\|e^{t\Delta}f\|_{\dot{C}^\beta_*}\lesssim t^{\frac{-(\beta-\gamma)}{2}}\sup_k  2^{\gamma k}\|P_kf\|_{L^\infty}\sim t^{\frac{-(\beta-\gamma)}{2}}\|f\|_{\dot{C}^\gamma_*}.\]
\end{proof}
\end{lemma}

The following lemma is taken from \cite[Lemma 2.8]{BCD11}, which applies to functions with Fourier support in any annulus. Compare this with Lemma \ref{Instantaneous dissipation rate}.

\begin{lemma}\label{cp lemma}
For $p\ge 2$, and a scalar function on $\T^d$ or $\R^d$ with $\spt\hat{f}\subset\{\xi:C^{-1}R\le|\xi|\le CR\}$, then
\begin{align}
\int(-\Delta f)\cdot f |f|^{p-2}\ge \frac{c}{p}R^2\int|f|^p,
\end{align}
with $c$ depending only on $C$.
\begin{proof}

Writing $f=\Div\nabla\Delta^{-1}f$, we find
\begin{align*}
\int |f|^p&=\int\Div\nabla\Delta^{-1}f\cdot f|f|^{p-2}\\
&=-(p-1)\int\nabla\Delta^{-1}f\cdot\nabla f|f|^{p-2}\\
&\le (p-1)\|\nabla\Delta^{-1}f\|_{L^p}\cdot\|\nabla f|f|^{\frac{p-2}{2}}\|_{L^2}\cdot\||f|^{\frac{p-2}{2}}\|_{L^{\frac{2p}{p-2}}}.
\end{align*}
by Hölder's inequality and $\frac{1}{p}+\frac{1}{2}+\frac{p-2}{2p}=1$. Observe that $\|\nabla\Delta^{-1} f\|_{L^p}\lesssim_C R^{-1}\|f\|_{L^p}$ and $\||f|^{\frac{p-2}{2}}\|_{L^{\frac{2p}{p-2}}}=\|f\|_{L^p}^{\frac{p-2}{2}}$. Simplifying the expression, we obtain:
\[\int|\nabla f|^2|f|^{p-2}\ge \frac{c}{(p-1)^2}R^2\int|f|^p.\]
On the other hand, we perform an integration by parts:
\[\int(-\Delta f)\cdot f |f|^{p-2}=(p-1)\int|\nabla f|^2|f|^{p-2}.\]
Combining the formulas, we conclude the proof.
\end{proof}
\end{lemma}

If we let $c_p$ be the actual optimal constant in the estimate, it is easy to show that $\lim_{p\to\infty}c_p=0$ if $C>\sqrt{3}$ in the lemma, i.e. the annulus is too thick, so that Lemma \ref{Thin annulus lemma} and Lemma \ref{Instantaneous dissipation rate} are meaningful. Specifically, the scalar function $f(x)=\frac{1}{8}(9\cos x-\cos 3x)$ defined on $\T=\R/2\pi\Z$  satisfies $\Delta f(x_0)=0$ for any point $x_0$ that attains the $L^\infty$-norm.  Then the probability measures
\[d\mu_p=\frac{|f|^p(x)}{\|f\|^p_{L^p}}\,dx\]
has a subsequence $d\mu_{p_q}$ converging weakly as $p_q\to\infty$ to a measure $d\mu_{\infty}$ supported in the set $\{x:|f|(x)=\|f\|_{L^\infty}\}$, so that
\[\int\frac{-\Delta f}{f}\,d\mu_{p_q}\to\int\frac{-\Delta f}{f}\,d\mu_{\infty}=0.\]
We can extract such a subsequence for every subsequence of $p\to\infty$, and thus conclude that $\lim_{p\to\infty}c_p=0$.

\begin{proof}[Proof of \eqref{C.1}, \eqref{C.2}]\label{Commutator estimate proof}

We prove formula \eqref{C.1}:
\[\perm(\nabla^{c+1},D_t^w)f=\sum C\tr \nabla^{c_0+1}D_t^{w_0}f\otimes \prod_{i\ge 1}\nabla^{c_i+1}D_t^{w_i}X,\]
where $c,w,c_i,w_i\ge0$, $\sum_{i\ge0}c_i=c$, and $w_0+\sum_{i\ge 1}(w_i+1)=w$. We prove this by induction on $c+w$.

When $c+w=0$, the only term is $\nabla f$. For the inductive step, applying $\nabla$ increases one of the indices $c_i$ by $1$, so we only need to examine the effect of applying $D_t$ to the expression. We use the following identity
\[D_t\nabla^{\tilde{c}}=\nabla^{\tilde{c}}D_t-\tr\sum_{\tilde{c}_{1}+\tilde{c}_2=\tilde{c}} \nabla^{\tilde{c}_1+1} X\otimes \nabla^{\tilde{c}_2},\]
which corresponds to increasing $w_i$ by $1$ or generating a new $c_i$ without changing the sum $\sum_{i\ge 0}c_i$. For $[\nabla^{c+1},D^w_t]f$, we proceed by induction on the same formula with the condition $w_0<w$. From the identity $[\nabla^{c+1},D_t^{w}]=[\nabla^{c+1},D_t]D_t^{w-1}+D_t[\nabla^{c+1},D_t^{w-1}]$, the condition $w_0<w$ follows readily.

For \eqref{C.2}:
\[\perm(Y\cdot\nabla,D_t^w)f=\sum C\tr\nabla D_t^{w_0}f\otimes D_t^{w_1}Y\otimes\bigotimes_{i\ge 2}\nabla D_t^{w_i}X,\]
where $w,w_i\ge 0$ and $w_0+w_1+\sum_{i\ge 2}(w_i+1)=w$.
We can write
\[\perm(Y\cdot\nabla,D_t^w)f=D_t^{\tilde{w}}(Y\cdot\nabla)D_t^{w-\tilde{w}}f.\]
We proceed by induction on $\tilde{w}$. The base case $\tilde{w}=0$ is trivial. Applying $D_t$ to the expression and using $D_t\nabla=\nabla D_t-\nabla X\cdot\nabla$ concludes the induction.

\end{proof}

\textbf{Acknowledgments.}
I would like to thank my advisor, Camillo De Lellis, for invaluable guidance and insightful feedback on this manuscript. The research of the author has been supported by the National Science Foundation
under grant no. DMS-2424441 and by the Simons Initiative on the Geometry of Flows
(Grant Award ID BD-Targeted-00017375, CF).

\nocite{*}
\bibliographystyle{amsalpha}
\bibliography{bibliography}

@article{Ise23,
  author  = {Isett, Philip},
  title   = {Regularity in time along the coarse scale flow for the incompressible {E}uler equations},
  journal = {Transactions of the American Mathematical Society},
  volume  = {376},
  number  = {10},
  pages   = {6927--6987},
  year    = {2023},
  doi     = {10.1090/tran/8991}
}

@incollection{CDS12,
  author    = {Conti, Sergio and {De Lellis}, Camillo and Sz\'{e}kelyhidi, Jr., L\'{a}szl\'{o}},
  title     = {$h$-principle and rigidity for {$C^{1,\alpha}$} isometric embeddings},
  booktitle = {Nonlinear Partial Differential Equations},
  series    = {Abel Symposia},
  volume    = {7},
  pages     = {83--116},
  publisher = {Springer, Heidelberg},
  year      = {2012},
  doi       = {10.1007/978-3-642-25361-4_4}
}

@book{Tri83,
  author    = {Triebel, Hans},
  title     = {Theory of Function Spaces},
  series    = {Monographs in Mathematics},
  volume    = {78},
  publisher = {Birkh\"{a}user Verlag},
  year      = {1983},
  address   = {Basel},
  isbn      = {3-7643-1381-1}
}

@book{Tri92,
  author    = {Triebel, Hans},
  title     = {Theory of Function Spaces {II}},
  series    = {Monographs in Mathematics},
  volume    = {84},
  publisher = {Birkh\"{a}user Verlag},
  year      = {1992},
  address   = {Basel},
  isbn      = {3-7643-2639-5}
}

@book{Ste70,
  author    = {Stein, Elias M.},
  title     = {Singular Integrals and Differentiability Properties of Functions},
  series    = {Princeton Mathematical Series},
  volume    = {30},
  publisher = {Princeton University Press},
  year      = {1970},
  address   = {Princeton, NJ}
}

@article{CDR20,
  author  = {Colombo, Maria and De Rosa, Luigi},
  title   = {Regularity in time of {H}\"{o}lder solutions of {E}uler and hypodissipative {N}avier--{S}tokes equations},
  journal = {SIAM Journal on Mathematical Analysis},
  volume  = {52},
  number  = {1},
  pages   = {221--238},
  year    = {2020},
  doi     = {10.1137/19M1243750}
}

@incollection{Ise25,
  author    = {Isett, Philip},
  title     = {Regularity of Trajectories and Smooth Observables in Rough {E}uler Flows},
  booktitle = {Partial Differential Equations: Waves, Nonlinearities and Nonlocalities},
  editor    = {Ehrnstr\"{o}m, Mats and Holden, Helge and Jakobsen, Espen R.},
  series    = {Abel Symposia},
  volume    = {18},
  pages     = {185--204},
  publisher = {Springer, Cham},
  year      = {2025},
  doi       = {10.1007/978-3-031-91282-5_9}
}

@book{Fri95,
  author    = {Frisch, Uriel},
  title     = {Turbulence: The legacy of {A. N. Kolmogorov}},
  publisher = {Cambridge University Press},
  address   = {Cambridge},
  year      = {1995},
  isbn      = {978-0-521-45713-2}
}

@book{LL87,
  author    = {Landau, L. D. and Lifshitz, E. M.},
  title     = {Fluid mechanics},
  edition   = {2nd},
  series    = {Course of Theoretical Physics, Volume 6},
  publisher = {Pergamon Press},
  address   = {Oxford},
  year      = {1987}
}

@article{CDF20,
  author    = {Colombo, Maria and {De Rosa}, Luigi and Forcella, Luigi},
  title     = {Regularity results for rough solutions of the incompressible {E}uler equations via interpolation methods},
  journal   = {Nonlinearity},
  volume    = {33},
  number    = {9},
  pages     = {4818--4836},
  year      = {2020},
  doi       = {10.1088/1361-6544/ab8d54}
}

@article{CVW15,
  author    = {Constantin, Peter and Vicol, Vlad and Wu, Jiahong},
  title     = {Analyticity of {L}agrangian trajectories for well posed inviscid incompressible fluid models},
  journal   = {Advances in Mathematics},
  volume    = {285},
  pages     = {352--393},
  year      = {2015},
  doi       = {10.1016/j.aim.2015.05.019}
}

@article{WML23,
  author    = {Wang, Yanqing and Mei, Xue and Liu, Jitao},
  title     = {{H\"{o}lder} regularity in time of solutions to the generalized surface quasi-geostrophic equation},
  journal   = {Applied Mathematics Letters},
  volume    = {137},
  pages     = {108480},
  year      = {2023},
  doi       = {10.1016/j.aml.2022.108480}
}

@article{Her60,
  author  = {Herz, Carl S.},
  title   = {The spectral theory of bounded functions},
  journal = {Transactions of the American Mathematical Society},
  volume  = {94},
  number  = {2},
  pages   = {181--232},
  year    = {1960}
}

@article{Ham86,
  author  = {Hamilton, Richard S.},
  title   = {Four-manifolds with positive curvature operator},
  journal = {Journal of Differential Geometry},
  volume  = {24},
  number  = {2},
  pages   = {153--179},
  year    = {1986}
}

@book{BCD11,
  author    = {Bahouri, Hajer and Chemin, Jean-Yves and Danchin, Rapha\"{e}l},
  title     = {{F}ourier analysis and nonlinear partial differential equations},
  series    = {Grundlehren der mathematischen Wissenschaften},
  volume    = {343},
  publisher = {Springer},
  address   = {Berlin, Heidelberg},
  year      = {2011},
  doi       = {10.1007/978-3-642-16830-7}
}

@article{CET94,
  author  = {Constantin, Peter and E, Weinan and Titi, Edriss S.},
  title   = {{O}nsager's conjecture on the energy conservation for solutions of {E}uler's equation},
  journal = {Communications in Mathematical Physics},
  volume  = {165},
  number  = {1},
  pages   = {207--209},
  year    = {1994},
  doi     = {10.1007/BF02099744}
}

@article{Kol41,
  author  = {Kolmogorov, A. N.},
  title   = {The local structure of turbulence in incompressible viscous fluid for very large {R}eynolds numbers},
  journal = {Doklady Akademii Nauk SSSR},
  volume  = {30},
  number  = {4},
  pages   = {299--303},
  year    = {1941}
}

@article{DS12,
  author  = {{De Lellis}, Camillo and Sz\'{e}kelyhidi, Jr., L\'{a}szl\'{o}},
  title   = {Dissipative continuous {E}uler flows},
  journal = {Inventiones Mathematicae},
  volume  = {193},
  number  = {2},
  pages   = {377--407},
  year    = {2012},
  doi     = {10.1007/s00222-012-0429-9}
}

@article{Ise18,
  author  = {Isett, Philip},
  title   = {A proof of {O}nsager's conjecture},
  journal = {Annals of Mathematics},
  volume  = {188},
  number  = {3},
  pages   = {871--963},
  year    = {2018},
  doi     = {10.4007/annals.2018.188.3.4}
}

\end{document}